\documentclass{amsart}
\usepackage{amsmath}
\usepackage{amssymb}
\usepackage{tikz-cd}
\usepackage{amsthm}
\usepackage{mathrsfs}
\usepackage{fullpage}
\usepackage{float}

\newcommand{\R}{\mathbb{R}}

\newcommand{\Z}{\mathbb{Z}}

\newcommand{\h}{\mathbb{H}}

\renewcommand{\int}[1]{#1_{\text{int}}}
\newcommand{\ext}[1]{#1_{\text{ext}}}

\newtheorem{thm}{Theorem}[subsection]
\renewcommand{\thethm}{%
	\ifnum\value{subsection}>0
	\thesubsection
	\else
	\thesection
	\fi
	.\arabic{thm}%
}
\newtheorem{lemma}[thm]{Lemma}

\newtheorem{cor}[thm]{Corollary}

\newtheorem{prop}[thm]{Proposition}

\newtheorem{question}[thm]{Question}

\theoremstyle{definition}
\newtheorem{Def}[thm]{Definition}

\theoremstyle{remark}
\newtheorem{rk}[thm]{Remark}

\title{Metric Spaces of Arbitrary Finitely-Generated Scaling Group \vspace{-2mm}}
\date{\today}
\author{Daniel Levitin}

\begin{document}
	
	\maketitle
	
	For a metric space $X$ with a compatible measure $\mu$, Genevois and Tessera defined the Scaling Group of $(X,\mu)$ as the subgroup $\Gamma$ of $\mathbb{R}_{>0}$ of positive real numbers $\gamma$ for which there are quasi-isometries of $X$ coarsely scaling $\mu$ by a factor of $\gamma$ \cite{GenevoisTessera2}. We show that for any finitely generated subgroup $\Gamma$ of $\mathbb{R}_{>0}$ there exists a space $N_\Gamma$, bi-Lipschitz equivalent to a graph of finite degree, with scaling group $\Gamma$. 
	
	\section{Introduction}
	
	If $(X,d_X)$ and $(Y,d_Y)$ are metric spaces, the function $q:X\to Y$ is a \textit{K quasi-isometry} if for all pairs of points $a,b\in X$, we have $\frac{1}{K}d_X(a,b)-K\le d_Y(q(a),q(b))\le Kd_X(a,b)+K$, and if the $K$-neighborhood of the image of $X$ covers $Y$, i.e. $\mathscr{N}_K(f(X))=Y$. Such a map $q$ has a \textit{coarse inverse} which we denote $\widetilde{q}:Y\to X$, i.e. a map $\widetilde{q}$ that is a quasi-isometry and so that $d_{\sup} (q\circ \widetilde{q},Id_Y)<\infty$ and $d_{\sup} (\widetilde{q}\circ q, Id_X)<\infty$. A quasi-isometry with $K=0$ is a \textit{bi-Lipschitz equivalence}. The existence of coarse inverses implies that quasi-isometry is an equivalence relation.
	
	The program of studying infinite-diameter metric spaces, and especially groups, up to quasi-isometry goes back to Gromov \cite{Gromov}. One general question of interest to Gromov was what role the additive factor plays, and therefore how different the equivalences of quasi-isometry and bi-Lipschitz could be.
	
	A natural setting in which to consider these questions is that of uniformly discrete spaces of bounded geometry. \textit{Uniformly discrete} metric spaces $(X,d)$ are those for which $\inf \{d(x_1,x_2): x_1, x_2\in X\}>0$. Metric spaces of \textit{bounded geometry} are those for which $\sup\{|B(x,r)|:x\in X\}<\infty$ for all $r$. We will abbreviate a space with both properties to be \textit{UDBG}. Typical examples are finitely-generated groups with word metrics, or more generally the vertex set of any graph of bounded degree. In this setting, Whyte showed that quasi-isometry and bi-Lipschitz equivalence coincide in the for non-amenable spaces \cite{Whyte}. This result was known prior in some restricted cases, see e.g. \cite{Nekrashevych}. Further examples of spaces where quasi-isometry and bi-Lipschitz equivalence agree were given in \cite{PapasogluWhyte}. On the other hand, in \cite{Dymarz1, Dymarz2} Dymarz gave the first examples of groups that are quasi-isometric but not bi-Lipschitz equivalent.
	
	Whyte's proof in \cite{Whyte} relied on recognizing that any bijective quasi-isometry of UDBG spaces is automatically a bi-Lipschitz equivalence. By framing bi-Lipschitz equivalences of UDBG spaces as the quasi-isometries for which every point has a single pre-image, we turn the process of checking whether a quasi-isometry is a bi-Lipschitz equivalence into a counting problem.  More, generally, a key observation in the study of quasi-isometries of UDBG spaces is the following. For UDBG spaces, we may take the point count of a set as a notion of volume, and consider how different maps scale this volume. Since quasi-isometries can expand or contract neighborhoods only by a fixed amount, the assumption of bounded geometry puts a maximum on how much a given quasi-isometry may stretch or shrink the volume. In this phrasing, bi-Lipschitz equivalences are volume-preserving quasi-isometries.
	
	In this paper, we will be interested not just in bi-Lipschitz equivalences, but in the slightly larger class of quasi-isometries which scale cardinalities by a constant on average.
	
	\begin{Def}[Coarsely $\gamma$-to-1]\label{gamma-scaling}
		A quasi-isometry $q:X\to Y$ between UDBG spaces is \textit{coarsely $\gamma$-to-1} if for all finite sets $S\subset Y$, and sufficiently large $r$, we have the following bound for some global constant $C_r$:
		$$\bigl||q^{-1}(S)|-\gamma|S|\bigr|\le C_r|\partial_r (S)|$$
		We will also call such a map $\gamma$\textit{-scaling}, or \textit{scaling} if it is $\gamma$-scaling for some $\gamma$.
	\end{Def}
	
	This formulation of the definition is due to Genevois and Tessera in their recent paper \cite{GenevoisTessera2}, going back to the ideas of Dymarz in \cite{Dymarz1}. Genevois and Tessera used scaling quasi-isometries to classify the quasi-isometries of certain wreath products \cite{GenevoisTessera1}. Note that in \cite{GenevoisTessera2}, this definition is extended to work in the setting of \textit{metric measure spaces}, which are metric spaces with a locally-finite Borel measure in which metric balls are relatively compact. Since we restrict to UDBG spaces with the counting measure, the above definition is a special case of the definition in \cite{GenevoisTessera2}. See Section \ref{misc} for more on arbitrary scaling groups in the setting of metric measure spaces.
	
	As an example of such a map, the inclusion of an index-$n$ subgroup into a finitely-generated group with word metric is coarsely $\frac{1}{n}$-to-1. However, a typical quasi-isometry should not be assumed to be scaling for any $\gamma$. For instance, the map $q:\Z\to\Z$ given by $$q(n)=\begin{cases} n & n\ge 0 \\ 2n & n<0\end{cases}$$ is a quasi-isometry, but it is not scaling \cite{GenevoisTessera2}.
		
	As with maps that are (exactly) $\gamma$-to-1, the scaling values are multiplicative under composition of maps. Moreover, if $q$ is $\gamma$-scaling, then its coarse inverse is $\frac{1}{\gamma}$-scaling. Therefore, for any metric space $X$, there is a group of scaling self quasi-isometries of $X$ \cite{GenevoisTessera2}. If we forget about the maps and consider only the scaling values, we get the \textit{scaling group} of $X$, denoted $Sc(X)$, consisting of positive real numbers $\gamma$ so that there is some $\gamma$-scaling quasi-isometry $q:X\to X$. 
	
	In the original setting of Cayley graphs, relatively few scaling groups have been computed. For lattices in Carnot groups, which are Lie groups that come with a continuous family of dilations already, the dilations give rise to a scaling group of all of $\R_{>0}$. The same is true for lattices in Sol and certain higher-rank analogues \cite{GenevoisTessera2}. Since the solvable Baumslag-Solitar groups $BS(1,n)$ for $n\ge 2$ have the model geometry of a tree of hyperbolic planes with a global horocyclic direction for all of the planes, the dilations of this horocyclic direction achieve all of $\R_{>0}$ as scaling values (see \cite{GenevoisTessera2} for a different treatment of this proof). The recent work of Genevois and Tessera on rigidity of wreath products over finitely presented and one-ended groups shows that this entire class has trivial scaling group \cite{GenevoisTessera1}. The only intermediate case that is known is Dymarz's theorem that certain wreath products over $\Z$ have scaling group is generated by a finite collection of primes \cite{Dymarz2}.
	
	\begin{question} [\cite{GenevoisTessera2} Question 7.6]
		Which finite sets $S$ of real numbers can be the generating sets of $Sc(G)$ for $G$ a group?
	\end{question}
	
	Relaxing to the setting of UDBG spaces, we provide a full answer to this question.
	
	\begin{thm}[Main Theorem]\label{MainThm}
		For any finite set $\{\gamma_i\}$ of real numbers, there is a space $N_\Gamma$ so that \newline $Sc(N_\Gamma)=\Gamma=\langle\{\gamma_i\} \rangle$
	\end{thm}
	
	The spaces $N_\Gamma$ are far from being Cayley graphs. They do not necessarily have any isometries, and certainly do not have a transitive isometry group. They can however be taken to be bi-Lipschitz equivalent to graphs whose degree is finite, and bounded above by a function of the number of generators.
	
	Although this paper will exclusively use Genevois and Tessera's formulation of scaling quasi-isometries, all of the work in \cite{Whyte, PapasogluWhyte, Dymarz1, Dymarz2} was phrased in terms of the Uniformly Finite Honology of Block and Weinberger \cite{BlockWeinberger}. These formulations are equivalent, and we sketch here the correspondence between the two. The uniformly finite homology provides a family of vector spaces as a quasi-isometry invariant to a UDBG space $X$. The group of self-quasi-isometries $QI(X)$ (typically an intractable group to work with) therefore has a family of representations given by the induced maps on the homology vector spaces. In the $0^{th}$ homology of an amenable space there is a nonzero fundamental class $[X]$. Scaling quasi-isometries are those with the fundamental class as an eigenvector, and the scaling group is the group of eigenvalues. Scaling maps $q:X\to Y$ between spaces are those for which $q_*[X]=\gamma[Y]$. In this setting, for instance, scaling quasi-isometries and the scaling group are both evidently groups under composition and multiplication respectively. In fact, one sees, e.g., that the collection of such $\gamma$ so that a $\gamma$-scaling map $q:X\to Y$ exists is evidently invariant under multiplication by both $Sc(X)$ and $Sc(Y)$. For more details on this viewpoint, see \cite{NowakYu}.
	
	\subsection*{Outline}
	
	Section \ref{Preliminaries} is devoted to basic definitions and simple properties that will be used throughout the paper.
	
	Section \ref{Rank1Section} deals with the simplest form of the main construction: achieving the scaling group $\langle \gamma\rangle$ for a single positive real number $\gamma$. In order to construct a space with such a scaling group, we start by considering a net $N$ in the Lie group Sol $=\R^2\rtimes\R$, which we give coordinates $(x,y,t)$. This case is treated separately from the cases with more generators because it is slightly simpler while providing the outline for all the proofs in the general case. It has the added advantage that the group Sol is 3-dimensional, which simplifies calculations and allows for supporting pictures in a more familiar context.
	
	Subsection 3.1 deals with generalities about the net $N$, and introduces a map $\rho:\text{Sol}\to N$ that allows quasi-isometries $\overline{q}$ of Sol to be approximated by quasi-isometries $q$ of $N$ and quasi-isometries $q$ of $N$ to induce quasi-isometries $\overline{q}$ of Sol. Subsection 3.2 uses $\rho$ to describe quasi-isometries of $N$ in terms of the classification of quasi-isometries of Sol due to Eskin-Fisher-Whyte.
	
	\begin{thm}[\cite{EFW1, EFW2}]
		
		Any quasi-isometry of Sol is at finite distance from a composition of an isometry with a map $(x,y,t)\mapsto (f_1(x), f_2(y), t)$ where $f_1$ and $f_2$ are bi-Lipschitz homeomorphisms of $\R$.
		
	\end{thm}
	
	See the complete statement in Theorem \ref{EFWMainThm}. In subsection 3.2 we prove the following proposition.
	
	\begin{prop}
		
		Let $q:N\to N$ be a scaling quasi-isometry, inducing a map $\overline{q}$ on Sol. Then the coordinate functions $f_1$ and $f_2$ of $\overline{q}_1$ are affine.
		
	\end{prop}
	
	A more precise statement is in proposition \ref{SolScalingIsAffineProp}. Along the way, we prove that the same holds in Sol, where scaling is defined in terms of the Haar measure. This is not used in the sequel, but may be of independent interest.
	
	Subsection 3.3 constructs the space $N_{\langle \gamma\rangle}$ whose scaling group is $\langle \gamma\rangle$. Roughly, this space arises from attaching an infinite family of flats $\{\Z^2_i:i\in\Z\}$ onto $N$ along a locus $A=\{a_i\}$, where the $a_i$ are spaced out so that their $x$ and $y$ coordinates differ by factors close to $\gamma$. The purpose of the flats is to make sure that any quasi-isometry must coarsely permute the $a_i$ (see Proposition \ref{FlatPreservingProp}). Along the way to prove Proposition \ref{FlatPreservingProp} we use the results of Eskin-Farb \cite{EskinFarb} or Kleiner-Leeb \cite{KleinerLeeb} to deduce the following well-known corollary that, to our knowledge, has not been proved in print.
	
	\begin{cor}
		There is no quasi-isometrically embedded flat in Sol of rank higher than 1.
	\end{cor}

	In Subsection 3.4, we assemble the pieces from the previous subsections, by observing that, in order to coarsely permute the $a_i$, a quasi-isometry of Sol with affine coordinate functions $f_1(x)=m_1x+b_1$ and $f_2(y)=m_2y+b_2$ must have $m_1$ and $m_2$ powers of $\gamma$. This shows that such a quasi-isometry must scale by a power of $\gamma$, so that $Sc(N_{\langle \gamma\rangle})\subset\langle\gamma\rangle$. We complete the 1-generator case of the main theorem by constructing an explicit quasi-isometry with scaling value $\gamma$.
	
	Section 4 is more terse than Section 3, and is devoted to generalizing the previous arguments to cases with more generators. It has four Subsections analogous to those of Section 3. The same approach works with the following changes. We use a net $N$ in a higher-rank solvable Lie group of the form $G=\R^{2n}\rtimes \R^{2n-1}$, where $n$ is the desired number of generators of $\Gamma=\langle \{\gamma_i\}_{i=1}^n\rangle$. We leverage a classification of their quasi-isometries due to Peng, which is described more fully in Theorem \ref{PengMainThm}.
	
	\begin{thm}[\cite{Peng1, Peng2}]
		The quasi-isometries of these Lie groups are compositions of isometries, bi-Lipschitz maps on the $\R^{2n}$ coordinates, and certain coordinate interchanges.
	\end{thm}
	
	As before, we use the $2n$ dimensions in the $\R^{2n}$ factor of $G$ to create $n$ loci for the attachment of flats. Each locus is spaced by a factor of a different generator. The resulting space is again $N_\Gamma$. The flats are chosen to be of high enough dimension so that they cannot be quasi-isometrically embedded into $G$. We prove in analogy to the Sol case that scaling quasi-isometries of the net have affine coordinate functions, and that any quasi-isometry of $N_\Gamma$ coarsely preserves the attachment points. The spacing of the flats then forces the affine coordinate maps to have linear terms that are powers of $\gamma_i$. 
	
	We conclude the paper in Section 5 by collecting miscellaneous improvements to the main construction. In particular, we show that the spaces constructed above can be taken to be graphs with edge-path distance. 
	
	\subsection*{Acknowledgments}
	
	I would like to thank my advisor, Tullia Dymarz, for encouraging me to tackle this project and for numerous useful conversations along the way. I would also like to thank the referee for many suggestions for how to improve this paper.
	
	\newpage
	
	\section{Preliminaries}\label{Preliminaries}
	
	We reiterate the definition of a quasi-isometry.
	
	\begin{Def} \label{QuasiIsometryDef}
		If $(X,d_X)$ and $(Y,d_Y)$ are metric spaces, the function $q:X\to Y$ is a \textit{K quasi-isometry} if for all pairs of points $a,b\in X$, we have $\frac{1}{K}d_X(a,b)-K\le d_Y(q(a),q(b))\le Kd_X(a,b)+K$, and if in addition, $q(X)$ is $K$-coarsely dense in $Y$. A map $q$ is a \textit{quasi-isometry} if it is a $K$ quasi-isometry for some $K$. 
	\end{Def}
	
	We will usually consider quasi-isometries up to the following equivalence.
	
	\begin{Def}
		
		Given two maps $f: (X, d_X)\to (Y, d_Y)$ and $g:(X, d_X)\to (Y, d_Y)$, the \textit{Supremum Distance} between them, denoted $d_{\text{sup}}(f, g)$, is defined to be $\sup_{x\in X} d_Y(f(x), g(x))$. We will denote $f\approx g$ when $d_{\text{sup}}(f,g)$ is finite.
		
	\end{Def}
	
	It is well known that the $\{\text{self quasi-isometries of X}\}/\approx$ is a group under the composition operation, denoted $QI(X)$.
	
	\begin{Def}\label{UDBGDef}
		
		A metric space $(X,d)$ is said to be \textit{uniformly discrete} if $\inf_{x_1\ne x_2} d(x_1,x_2)>0$. A metric space is said to have \textit{bounded geometry} if for each $r\in\R_{>0}$, there is some finite $B_r$ so that \newline $\sup_{x\in X} |B(x,r)|=B_r$. A space with both properties will be called \textit{UDBG}.
	\end{Def}
	
	We will use the following remark routinely and not always explicitly
	
	\begin{rk} \label{QIBoundedtoOne}
		
		It follows immediately from the definition of a $K$ quasi-isometry that the pre-image of a point is contained in a ball of radius $K^2+K$. In a setting with bounded geometry, the pre-image of a point therefore contains no more than $B_{K^2+K}$ points.
		
	\end{rk}
	
	\begin{Def}\label{AmenableDef}
		Denote $\partial_r(S)=\mathscr{N}_r(S)\cap\mathscr{N}_r(S^c)$. A UDBG space $X$ will be said to be \textit{amenable} if it admits a sequence of finite subsets $S_i\subset X$  so that for each positive $r$,  $\lim_{i\to\infty} \frac{|\partial_r (S_i)|}{|S_i|}=0$. Such a sequence is termed a \textit{F\o lner sequence}.
	\end{Def}

	We will extend this definition for Lie groups by replacing cardinalities with Haar measures. That is, such a group is amenable if we have a sequence $S_i$ of measurable sets so that $\lim_{i\to\infty} \frac{\mu(\partial_n(S_i))}{\mu(S_i)}\to0$.
	
	\begin{Def}
		
		Let $q:X\to Y$ be a quasi-isometry of UDBG spaces. If there are constants $C$ and $r$ depending only on $q$ so that $$\bigl||q^{-1}(S)|-\gamma|S|\bigr|\le C|\partial_r(S)|$$ for every finite subset $S$ of $Y$, then we will say that $q$ is \textit{coarsely} $\gamma$\textit{-to-1}, or $\gamma$\textit{-scaling}.
		
		For a UDBG space $X$ the collection of $k$ so that there exists a coarsely $\gamma$-to-1 self quasi-isometry of $X$ is termed the \textit{Scaling Group of X}, denoted $Sc(X)$. It is a multiplicative group of positive real numbers.
		
	\end{Def}
	
	Because it is only for F\o lner sequences $S_i$ that $|\partial_r(S_i)|$ be made arbitrarily small relative to $|S_i|$, we will usually only consider scaling quasi-isometries and the scaling group for amenable spaces. If we did not, it is a result of Whyte that the scaling group of any non-amenable space is all of $\R_{>0}$, and in fact that every quasi-isometry is $\gamma$-scaling for every value of $\gamma$ \cite{Whyte}. When we wish to study a specific F\o lner sequence $S_i$, we will say that $q$ is $\gamma$\textit{-scaling on} $S_i$ to mean that there are positive constants $C$ and $r$ so that $$\bigl||q^{-1}(S_i)|-\gamma|S_i|\bigr|\le C|\partial_r(S_i)|$$
	
	The following theorem is an easy corollary to and reformulation of an observation made by Block and Weinberger in \cite{BlockWeinberger}, and first proved explicitly by Whyte in theorem 7.6 of \cite{Whyte}.
	
	\begin{thm} \cite{BlockWeinberger, Whyte}  \label{BWScalingCriterion}
		If $q:X\to Y$ is a quasi-isometry of amenable UDBG spaces, then $q$ is $\gamma$-scaling iff $q$ is $\gamma$-scaling on all F\o lner sequences $S_i$ in $Y$.
	\end{thm}
	
	\begin{rk} \label{ScalingImpliesRatio}
		
		Observe that if $q$ is $\gamma$-scaling on $S_i$, then $\lim_{i\to\infty} \frac{|q^{-1}(S_i)|}{|S_i|}=\gamma$. However, the converse is not true in general since $\bigl||q^{-1}(S_i)|-\gamma|S_i|\bigr|$ could grow faster than $\partial_r(S_i)$ for any $r$, but still slower than $|S_i|$.
		
	\end{rk}
	
	As a consequence of Theorem \ref{BWScalingCriterion}, if $X$ is amenable, then a quasi-isometry $q$ is $\gamma$-scaling for at most one value of $\gamma$. It is typically impractical to check every F\o lner sequence in a space $Y$, as in general they may have very little structure. However, the contrapositive is very useful. In particular, if $q$ is $\gamma_1$-scaling on a F\o lner sequence $S_{i,1}$ and $\gamma_2$-scaling on a F\o lner sequence $S_{i, 2}$, then $q$ is not scaling. Remark \ref{ScalingImpliesRatio} says that it suffices to show that $\lim_{i\to\infty} \frac{|q^{-1}(S_{i,1})|}{|S_{i,1}|}\ne \lim_{i\to\infty} \frac{|q^{-1}(S_{i,2})|}{|S_{i,2}|}$ to show that $q$ is not scaling.
	
	We record also the following lemma
	
	\begin{lemma} \label{ApproxInvarianceLemma}
		Let $q_1$ and $q_2$ be quasi-isometries between amenable metric spaces $X$ and $Y$, and let $q_1\approx q_2$. Then $q_1$ is $\gamma$-scaling iff $q_2$ is.
	\end{lemma}
	
	\begin{proof}
		Suppose $q_1$ is $\gamma$-scaling, $d_{\text{sup}}(q_1,q_2)=D$, and let $S_i$ be a F\o lner sequence in $Y$. From Theorem \ref{BWScalingCriterion},  $\bigl||q_1^{-1}(S_i)-\gamma|S_i|\bigr|=O(|\partial_r(S_i)|)$. Then if $x$ is a point so that $q_1(x)$ is in $S_i$ and $q_2(x)$ is not, or vice versa, $q_1(x)$ is within $D$ of $q_2(x)$, and thus within $D$ of both of $S_i$ and its complement. Then $q_1^{-1}(S_i)\Delta q_2^{-1}(S_i)\subset q_1^{-1}\partial_D(S_i)$. Applying Remark \ref{QIBoundedtoOne}, the two sets $q_1^{-1}(S_i)$ and $q_2^{-1}(S_i)$ differ in cardinality by $O(|\partial_D(S_i)|)$. After possibly increasing the value of $r$, we determine that $\bigl||q_2^{-1}(S_i)-\gamma|S_i|\bigr|=O(|\partial_r(S_i)|)$ and the lemma follows from Theorem \ref{BWScalingCriterion}. The reverse direction follows from  swapping the roles of $q_1$ and $q_2$.
	\end{proof}

	This lemma shows that we typically need not worry which representative of a given equivalence class in $QI(X)$. We will frequently apply maps which move every point a bounded distance without explicitly pointing this out.
	
	The following technical lemma describes approximate commutativity between the operations $q^{-1}$ and $\partial$ in the case of surjective quasi-isometries. This will help us achieve certain upper bounds on error terms. An analogous statement could be made to achieve a lower bound, but we will not need lower bounds on error terms in this paper.
	
	\begin{lemma}\label{SurjectiveQIs}
		
		Let $q:X\to Y$ be a surjective $K$ quasi-isometry and $S\subset Y$. Then $$q^{-1}(\partial_r (S))\subset \partial_{Kr+K} (q^{-1}(S))$$
		
	\end{lemma}
	
	\begin{proof}
		
		Let $y\in \partial_r(S)$, and $q(x)=y$. Let $y_1\in S$, $y_2\in S^c$, and $d_Y(y, y_i)<r$ for $i=1, 2$. Then if we pick points $x_1$ and $x_2$ so that $q(x_i)=y_i$, we see that $d(x, x_i)\le Kr+K$ for each $i$. Now, $x_1\in q^{-1}(S)$ and $x_2\in q^{-1}(S^c)$.  But for any function, $q^{-1}(S^c)=(q^{-1}(S))^c$, so that $x$ is within $Kr+K$ of both $q^{-1}(S)$ and its complement.
		
	\end{proof}
	
	Of course, a typical quasi-isometry is far from surjective.
	
	We conclude this section with a notational convention. We will at various points compute or refer to the existence of radii of neighborhoods and boundaries. This radius will always be denoted $r$. If the size of a neighborhood of boundary must be multiplied by a constant, that constant will be denoted $C$. Between portions of the paper, the values of these constants may change a finite number of times. Their exact numerical values will never matter, however.	
	
	\section{Sol and the 1-generator case}\label{Rank1Section}
	
	\subsection{F\o lner sets and Nets}
	
	We will begin with the case of a subgroup $\{\gamma^n: n\in\Z\} \subset\R^+$, for which we will construct a space based on a net in the Lie group Sol, which we will refer to as $G$ throughout this section.
	
	\begin{Def}
		Sol is topologically $\R^3$ with the group law  $$(a_1,b_1,c_1)(a_2,b_2,c_2)=(a_1+e^{c_1}a_2,b_1+e^{-c_1}b_2,c_1+c_2)$$ One may check directly that there is a normal subgroup $G_1=\{(a,b,0): a, b\in \R\}$ and a subgroup \newline $G_2=\{(0,0,c): c\in \R\}$ so that Sol is given by $G_1\rtimes G_2$. The action of $(0,0,c)$ on $G_1$ is given by $e^{c\alpha}$ where $\alpha$ is the matrix $\begin{bmatrix} 1 & 0\\ 0 & -1\end{bmatrix}$. For a point $(a,b,c)$ we will refer to $c$ as the \textit{height} or the $\textit{t-coordinate}$, and $a$ and $b$ as the $\textit{x-coordinate}$ and $\textit{y-coordinate}$ respectively. One may check directly that the Riemannian metric $ds^2 = e^{-2t}dx^2+e^{2t}dy^2+dt^2$ is left invariant.
	\end{Def}
	
	For the rest of the section, $d_G$ denotes the path distance arising from this Riemannian metric, and $\mu$ denotes the Haar measure.
		
	\begin{rk}\label{FlatsInSolRk}
		
		One useful way to view Sol is that there is a quasi-isometric embedding $q:\text{Sol}\to \h^2\times\h^2$ as follows. If we give the two copies of $\h^2$ their upper half-plane coordinates, then $(x,y,t)\mapsto \bigl((x,e^{t}),(y,e^{-t})\bigr)$, so that the parameter space is $\R^4$. That is, if $\h^2\times\h^2=\{\bigl((x,e^{t_1}),(y,e^{t_2})\bigr)\}$, then Sol is quasi-isometric to the locus $t_1+t_2=0$ within $\h^2\times\h^2$. The $t$ direction in Sol is mapped to the $t_1-t_2$	 direction in this locus. Lines $L(s)=(a,b,s)$ in Sol are isometrically embedded, as are their images in $\h^2\times\h^2$. When we wish to refer to this, we will say that the \textit{t-direction is undistorted}. 
		
		From the viewpoint of this embedding into $\h^2\times\h^2$, we see that $d_G((x,y,t),(x+k,y,t))=O_t(\log(|k|))$, and similarly $d_G((x,y,t),(x,y+k,t))=O_t(\log(|k|))$. This is because of the growth of distances in the first and respectively second copy of $\h^2$. 
		
	\end{rk}
	
	With respect to its Haar measure $\mu$ and Riemanian path metric $d_G$, Sol is amenable. An example F\o lner sequence is any collection of \textit{box sets} of the form  $[0,n)\times [0,m)\times [-\log(m),\log(n))$ when $nm>1$ and $nm\to \infty$ \cite{DymarzNavas}.
	
	\begin{rk}\label{SolMeasureRk}
		In Sol, the Haar Measure is $dxdydt$ since $\alpha$ has determinant $1$. Therefore a box \newline $[a_1, a_2]\times [b_1,b_2]\times [c_1,c_2]$ has  volume $(a_2-a_1)(b_2-b_1)(c_2-c_1)$ as in $\R^3$. This equality then extends to all measurable sets by considering unions. However, unlike in $\R^3$, it is possible to construct F\o lner sequence fixing one of $m$ or $n$ and letting the other grow.
	\end{rk}
	
	We will now construct a discrete subset of Sol for which the counting measure will approximate the Haar measure. For that, we make the following definition.
	
	\begin{Def}\label{NetDef}
		Let $(X,d)$ be a metric space. A \textit{net} in $X$ is a discrete subset $N\subset X$ which is coarsely dense in $X$, i.e. so that there exists $D$ such that $d(x,N)<D$ for each $x\in X$. 
	\end{Def}
	
	We construct a uniformly discrete net following \cite{DymarzNavas}.
	
	\begin{Def}\label{NDef}
		Denote by $N$ the collection $\{t^{n_1}x^{n_2}y^{n_3}:n_i\in\Z\}$. Note that in coordinates, we express these points as $t ^{n_1}x^{n_2}y^{n_3}=(e^{n_1}n_2,e^{-n_1}n_3,n_1)$. We equip $N$ with the metric that is the restriction of $d_G$, which we will refer to as $d_N$ by abuse of notation. 
	\end{Def}
	
	For the remainder of this section, for a subset $A\subset G$, we will set the notation that $\partial^G_r(A)$ denotes the set of points within $r$ of both $A$ and $G\setminus A$. On the other hand, if $A$ is also a subset of $N$, then $\partial^N_r(A)$ denotes the set of net points within $r$ of both $A$ and $N\setminus A$. We reserve the symbol $\partial_r$ with no superscript for later.
	
	We now describe in what sense $N$ approximates Sol.
	
	\begin{rk}
		
		Though $N$ is not a subgroup, it nevertheless has some sense of a fundamental domain as follows.	Let $B$ denote the parameter box $[0,1)\times[0,1)\times[0,1)\subset G$. Translating $B$ by elements $t^{n_1}x^{n_2}y^{n_3}$ with $n_i$ integers gives a disjoint collection of isometric boxes $[e^{n_1}n_2,e^{n_1}n_2+1)\times[e^{-n_1}n_3,e^{-n_1}n_3+1)\times[n_1,n_1+1)$. We will refer to $t^{n_1}x^{n_2}y^{n_3}\cdot B$ as $B_{n_1, n_2,n_3}$ One may check directly that these boxes partition Sol. See Figure 1 for a picture of how these boxes look at different heights.
		
	\end{rk}
	
	\begin{figure}
				
		\centering
		\includegraphics{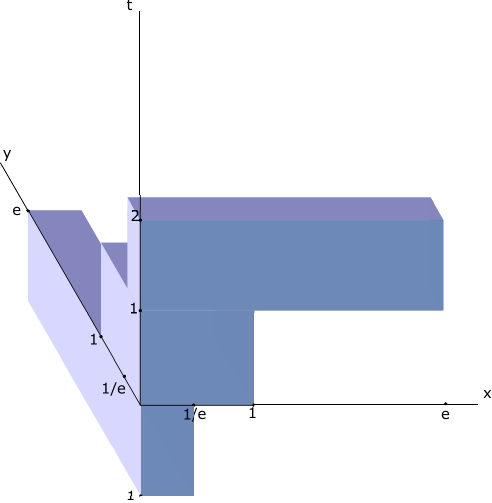}
		
		\caption{The special boxes $B_{1,0,0}$ (top) $B_{0,0,0}$ (middle), and $B_{-1,0,0}$ (bottom) in $(x,y,t)$-coordinates}
		
	\end{figure}
	
	We will call the $B_{n_1,n_2,n_3}$ \textit{special boxes} to distinguish them from other box sets. We will also denote by $D_B$ the diameter of $B$ and thus of all special boxes. We see immediately that $\mu(B)=1$, and therefore all the special boxes have measure $1$ by left-invariance of $\mu$.
	
	These boxes allow us to pass between quasi-isometries of $N$ and quasi-isometries of $G$ as follows. 
	
	\begin{Def}
		For any point $(a,b,c)$ in $G$, we define $\rho(a,b,c)$ to be $(e^{n_1}n_2,e^{-n_1}n_3,n_1)$ for the unique triple so that $(a,b,c)$ is in the special box $B_{n_1,n_2,n_3}$.
	\end{Def} 
	
	In concrete terms, $\rho$ sends every box in $G$ to the corner it contains. This map evidently moves points a finite distance. Precomposing with $\rho$ allows us to extend any quasi-isometry $q:N\to N$ to a new quasi-isometry $\overline{q}:G\to G$. Restricting to $N$ and postcomposing with $\rho$ allows us to approximate any quasi-isometry $\overline{q}:G\to G$ with one sending $N$ to $N$. Notice that $\rho\circ\overline{q}\approx\overline{q}\approx\overline{q}\circ\rho$. 
	
	The boxes will also allow us to approximate sets in convenient ways.
	
	\begin{Def}
		
		Given a subset $A$ of $G$, we define the set $\ext{A}$ to be 
		$$\ext{A}=\bigsqcup_{B_{n_1,n_2,n_3}\cap A\ne\emptyset} B_{n_1, n_2, n_3}$$
		and the set $\int{A}$ to be
		$$\int{A}=\bigsqcup_{B_{n_1,n_2,n_3}\subset A} B_{n_1,n_2,n_3}$$
		
		That is, $\ext{A}$ is the union of special boxes meeting $A$, and $\int{A}$ is the union of special boxes contained in $A$.
		
	\end{Def}
	
	We will routinely need the following comparisons.
	
	\begin{rk}\label{ExtAndInt}
		
		Let $A$ be a measurable subset of $G$. We have inclusions
		
		$$S\setminus \partial^G_{D_B}(A)\subset \int{A}\subset A\subset \ext{A}\subset A\cup\partial^G_{D_B}(A)$$
		
		As a result, if $A$ has finite measure, we have the inequalities of measures
		
		\begin{align*}
			&\mu(A)-\mu(\partial^G_{D_B}(A))\le\mu(A\setminus \partial^G_{D_B}(A))\le\mu(\int{A})\le\mu(A)\\
			&\mu(A)+\mu(\partial^G_{D_B}(A))\ge \mu(A \cup \partial^G_{D_B}(A))\ge \mu(\ext{A})\ge \mu(A)
		\end{align*}
		
		and the inequalities of cardinalities

		\begin{align*}
			&|A\cap N| - |\partial^G_{D_B}(A)\cap N| \le \bigl|\bigl(A\setminus\partial^G_{D_B}(A)\bigr)\cap N\bigr|\le|\int{A}\cap N|\le|A\cap N|\\
			&|A\cap N| +|\partial^G_{D_B}(A)\cap N|\ge \bigl|\bigl(A\cup \partial^G_{D_B}(A)\bigr)\cap N\bigr| \ge |\ext{A}\cap N| \ge |A\cap N|
		\end{align*}
		
		Moreover, since $\int{A}$ and $\ext{A}$ are disjoint unions of special boxes, each of which have volume $1$ and contain one point of $N$, it follows that $$\mu(\int{A})=|\int{A}\cap N|$$
		and
		$$\mu(\ext{A})=|\ext{A}\cap N|$$
		
		These equalities let us interchange between the two inequality chains above, e.g. we can conclude that $\mu(A)-\mu(\partial^G_{D_B}(A))\le \bigl|\bigl(A\cup \partial^G_{D_B}(A)\bigr)\cap N\bigr|$
				
	\end{rk}
	
	Since we will sometimes use Remark \ref{ExtAndInt} in cases where the set $A$ is itself a boundary, we will need the following lemma. Its proof is an easy exercise.
	
	\begin{lemma}\label{BoundaryAdditivity}
		
		Let $A$ be a subset of Sol, and $r_1$ and $r_2$ positive real numbers. Then 	
		
		$$\partial^G_{r_2}(A)\cup\partial^G_{r_1}\left(\partial^G_{r_2}(A)\right)\subset \partial^G_{r_1+r_2}(A)$$
		
	\end{lemma}
		
	We next prove that $N$ is an amenable net with bounded geometry. All of the required properties will follow from the isometric left action of $G$ on itself. We must be careful, however, because the net $N$ arises from the action of a set of elements of $G$ that is not a subgroup. So $N$ is invariant only under the left action of integer powers of $t$, which will turn out to suffice. If we had taken a lattice for $N$, these lemmas would be trivial, but we would encounter technical difficulties in Section 4 when generalizing to higher-rank Lie Groups. 
	
	\begin{lemma}\label{NPropertiesLemma}
		$N$ is a uniformly discrete net with bounded geometry.
	\end{lemma}
	
	\begin{proof}
		
		$B_{n_1,n_2,n_3}$ is isometric to $B$, so that each $B_{n_1,n_2,n_3}$ is contained in the $D_B$-neighborhood of $t^{n_1}x^{n_2}y^{n_3}$. This shows coarse density.
		
		Since the heights of net points are integers, and the $t$-direction is undistorted, net points at different heights are a distance at least $1$ away from one another. If two net points are at the same height, then by left translating by an integer power of $t$, we may assume that they are at height $0$. It is simple to check that each point $(x,y,0)$ has the four equidistant points $(x\pm 1, y, 0), (x, y\pm1, 0)$ as its nearest neighbors among points at height $0$, which provides a lower bound for distances to points at the same height.
		
		Let $B(p,r)$ be a ball in $G$, where $p$ is a point in $N$. To bound $|B(p,r)\cap N|$, Consider $\ext{B(p,r)}$. By Remark \ref{ExtAndInt},
		
		$$|\ext{B(p,r)}\cap N|=\mu(\ext{B(p,r)})\le \mu\bigl(B(p, r)\cup\partial_{D_B}(B(p,r))\bigr)$$
		
		Because Sol is a path metric space, $B(p, r)\cup\partial_{D_B}(B(p,r))=B(p,r+D_B)$. Then since $\ext{B(p,r)}$ contains $B(p,r)$, it follows that $\left(B(p,r)\cap N\right)\subset \left(\ext{B(p,r)}\cap N\right)$. Therefore, we have 
		
		$$|B(p,r)\cap N|\le \mu(\ext{B(p,r)})\le \mu(B(p, r+D_B))$$
		
		The rightmost term is a constant independent of $p$ by the left-invariance of the Haar measure. This therefore bounds the cardinality of an $r$-ball in $N$ independent of its center, so that $N$ has bounded geometry.
		
	\end{proof}

	The proof that $N$ is amenable will rely on techniques that will be repeated several times in the paper, and therefore several lemmas are in order. The goal is to to intersect a F\o lner sequence from Sol with the net, and control the error that arises from switching from measures to cardinalities or vice versa.
	
	\begin{lemma} \label{VolumePtCountComparison}
		
		For any compact measurable subset $A\subset G$, 
		
		$$|\mu(A)-|A\cap N||\le \mu(\partial^G_{D_B} (A))$$
		
	\end{lemma}
	
	\begin{proof}
		
		Rearranging the second inequality in Remark \ref{ExtAndInt} yields
		
		$$-\mu(\partial^G_{D_B}(A))\le \mu(\int{A})-\mu(A)$$
		
		and equivalently
		
		$$\mu(A)-\mu(\int{A})\le \mu(\partial^G_{D_B}(A))$$
		
		Similarly,
		
		$$\mu(\ext{A})-\mu(A)\le\mu(\partial^G_{D_B}(A))$$
		
		From the last two inequalities in Remark \ref{ExtAndInt}, $$|\int{A}\cap N|=\mu(\int{A})\le \mu(A)\le \mu(\ext{A})=|\ext{A}\cap N|$$  
		
		Finally, since $\int{A}\subset A\subset \ext{A}$, it follows that $\int{A}\cap N\subset A\cap N\subset \ext{A}\cap N$. We therefore conclude that
		
		$$|A\cap N|-\mu(A)\le \mu(\ext{A})-\mu(A)\le\partial^G_{D_B}(A)$$
		
		and
		
		$$\mu(A)-|A\cap N| \ge \mu(A)-\mu(\int{A})\ge \partial^G_{D_B}(A)$$
		
	\end{proof}
	
	The second lemma will permit us to bound the cardinalities of an $N$-boundary by the measures of a slightly larger $G$-boundary.
	
	\begin{lemma}\label{NBoundarytoGBoundary}
		
		Let $A'$ be a finite subset of $N$, and $r$ a positive number, and $A$ a subset of $G$ so that $A\cap N=A'$. Then $|\partial^N_r(A')|\le\mu\bigl(\partial^G_{r+D_B}(A)\bigr)$
		
	\end{lemma}
	
	In particular the case $A=A'$ is allowed. However, this will mostly be relevant when applied to a F\o lner sequence of sets $A$ and their intersections with the net.
	
	\begin{proof}
		
		If $p$ is within $r$ of points $p_1\in A'$ and $p_2\in N\setminus A'$, then $p_1\in A$, while $p_2\in G\setminus A$ because $A\cap N=A'$. Thus $p$ is in $\partial^G_r(A)$. But then each special box containing such a point $p$ is within $r+D_B$ of $A$ and $G\setminus A$, and each such special box has measure 1. 
		
	\end{proof}

	\begin{cor} \label{NAmenabilityCor}
		Let $S_i$ be a F\o lner sequence for $G$. Then $S_i\cap N=S_i'$ is a F\o lner sequence for $N$. In particular, $N$ is amenable.
	\end{cor}
	
	\begin{proof}
		
		By Lemma \ref{VolumePtCountComparison},
		
		$$\frac{|\partial^N_r(S_i')|}{|S_i'|}\le \frac{|\partial^N_r(S_i')|}{|\mu(S_i)-\mu(\partial^G_{D_B}(S_i))|}=\frac{|\partial^N_r(S_i')|}{\mu(S_i)}\frac{1}{\left|1-\frac{\mu(\partial^G_{D_B}(S_i))}{\mu(S_i)}\right|}$$
		
		Applying Lemma $\ref{NBoundarytoGBoundary}$ gives 
		
		$$\frac{|\partial^N_r(S_i')|}{\mu(S_i)}\frac{1}{\left|1-\frac{\mu(\partial^G_{D_B}(S_i))}{\mu(S_i)}\right|}\le \frac{\mu\bigl(\partial^G_{r+D_B}(S_i)\bigr)}{\mu(S_i)}\frac{1}{\left|1-\frac{\mu(\partial^G_{D_B}(S_i))}{\mu(S_i)}\right|} $$
		
		The term $\frac{1}{\left|1-\frac{\mu(\partial^G_{D_B}(S_i))}{\mu(S_i)}\right|}$ goes to $1$ as $i\to\infty$, and the term $\frac{\mu\bigl(\partial^G_{r+D_B}(S_i)\bigr)}{\mu(S_i)}$ goes to $0$ as $i\to\infty$, in both cases since the $S_i$ form a F\o lner sequence in $G$. 
		
	\end{proof}
	
		We conclude the subsection with two lemmas that provide a partial converse to Lemma \ref{NBoundarytoGBoundary}.
	
		\begin{lemma}\label{GBoundarytoNBoundarySpecial}
		
		For each $r>0$, there is some $r_1=r_1(r)$ so that if $S$ is any union of special boxes \newline $S=\bigsqcup_{i=1}^k B_{n_{1,i},n_{2,i},n_{3,i}}$, then $\mu(\partial^G_r(S))\le|\partial^N_{r_1}(S\cap N)|$.
		
	\end{lemma} 
	
	\begin{lemma} \label{GBoundarytoNBoundaryBox}
		
		For each $r>0$ and $c>0$, there is some $r_1=r_1(r,c)$ so that the following holds. Let $c_1, c_2$ be real numbers and $|c_i|<c$. Then for any box $S=[a_0,a_1)\times[b_0,b_1)\times[-\log(b_1-b_0)+c_2, \log(a_1-a_0)+c_1)$ where $[-\log(b_1-b_0)+c_2, \log(a_1-a_0)+c_1)$ contains an integer, $\mu(\partial^G_r(S))\le|\partial^N_{r_1}(S\cap N)|$.
		
	\end{lemma}
	
	We will prove the two of these lemmas simultaneously, since the proofs will be essentially the same.
	
	\begin{figure}
		
		\centering
		\includegraphics{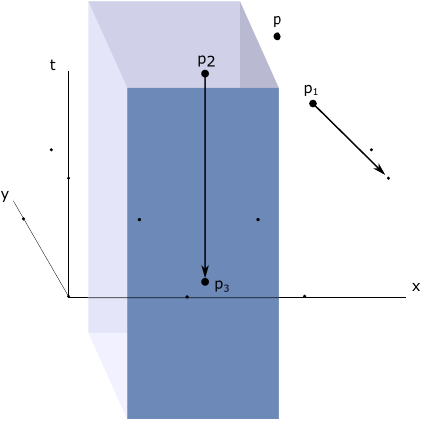} \hspace{1 cm} \includegraphics{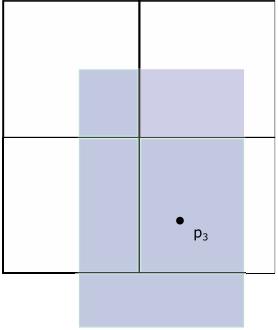}
		\caption{The proof of Lemma \ref{GBoundarytoNBoundaryBox}. On the left, bold points are $p$ and the $p_i$, and non-bold points are elements of $N$ other than $p$. The point $p_1$ is in the box $B_{1,0,0}$, and lies to the right of the box $S$. Therefore, it is moved to the net point $tx$ in the box $B_{1,1,0}$, which is even further to the right. $S$ is too thin in the $x$-direction to contain any net points at height $1$, so $p_2$ is moved down to height $0$ where $S$ is required to contain a net point. On the right is a view of the height-$0$ slice of the box $S$.}
		
	\end{figure}

	\begin{proof}
		
		Firstly, note that by Remark \ref{ExtAndInt} and Lemma \ref{BoundaryAdditivity},
		
		$$\mu(\partial^G_r(S))\le|\partial^G_{r+D_B}(S)\cap N|$$
		
		Therefore to get $r_1$ we must increase $r+D_B$ enough to guarantee that every net point $p$ in $\partial^G_{r+D_B}(S)$ also appears in $\partial^N_{r_1}(S\cap N)$. So let $p$ be within $r+D_B$ of points $p_1\in G\setminus S$ and $p_2\in S$.
		
		If $S$ is a union of special boxes, then $p_2$ not being in $S$ implies that $\rho(p_2)$ is not in $S$ either, so that there is a net point missing $S$ at distance at most $r+2D_B$ from $p$.
		
		If $S$ is instead a box, and $p_1=(x,y,t)$, then $x$, $y$, or $t$ lies outside of the interval defining $S$. Since $\rho$ reduces each coordinate, if any coordinate of $p_2$ is less than the defining interval of $S$, then $\rho(p_2)$ will be outside of $S$. $\rho(p_2)$ is at distance at most $r+2D_B$ from $p$. If instead some coordinate of $p_2$ is greater than the defining interval of $S$, suppose $p_2$ is in box $B_{n_1,n_2,n_3}$. Then the net points $t^{n_1+1}x^{n_2}y^{n_3}$, $t^{n_1}x^{n_2+1}y^{n_3}$, and $t^{n_1}x^{n_2}y^{n_3+1}$ all lie in the boundary of $B_{n_1,n_2,n_3}$, and are therefore at most $D_B$ from $p_2$. Whichever coordinate of $p_2$ lies above the defining interval of $S$, one of these three net points has a greater value of that coordinate than $p_2$, and therefore one of these three net points is outside $S$. So again we find a net point within $r+2D_B$ from $p$ that is outside $S$.
		
		On the other hand, we must move $p_2$ to a net point in $S$. If $S$ is a union of special boxes, then rounding $p_2$ yields this point, at distance at most $r+2D_B$ from $p$. Lemma \ref{GBoundarytoNBoundarySpecial} then holds with $r_1=r+2D_B$
			
		If $S$ is a box, then we wish to find a height at which $S$ is required to contain net points. Any such height is necessarily integral. At height $h$, the special boxes have width $e^{h}\times e^{-h}$. Therefore, at a height in $[-\log(b_1-b_0),\log (a_1-a_0))\cap [-\log(b_1-b_0)+c_2,\log (a_1-a_0)+c_1)$, the box $S$ is at least $(a_1-a_0)e^{-\log(a_1-a_0)}=1$ special box long in the $x$-direction and at least $(b_1-b_0)e^{-\log(b_1-b_0)}=1$ special box wide in the $y$-direction. See figure 2 for some example pictures. Since $|c_i|<c$, we conclude that $S$ contains net points at any integral height in the range $[-\log(b_1-b_0)+c, \log(a_1-a_0)-c)$. These heights exist by assumption.
		
		Since $p_2$ is at a height in $[-\log(b_1-b_0)+c_2, \log(a_1-a_0)+c_1)$, it is a most a distance $2c$ in the $t$-direction away from the interval $[-\log(b_1-b_0)+c, \log(a_1-a_0)-c)$.  Therefore moving $p_2$ vertically a distance at most $2c+1$ reaches a point $p_3$ at integral height $n$ in $S$. Suppose $p_3$ is in box $(a, b, n)$. As shown in Figure 2, $S$ contains the lower left corner of at least one of the four boxes $(a, b, n)$, $(a+1, b, n)$, $(a, b+1, n)$ or $(a+1, b+1, n)$. All such corners are on the boundary of box $(a,b,n)$ so that they are at most $D_B$ from $p_3$. Therefore, $p$ is a distance at most $r+2c+1+2D_B$ from a point in $S\cap N$. Lemma \ref{GBoundarytoNBoundaryBox} then follows with $r_1=\max\{r+2c+1+2D_B,r+2D_B\}$.
		
	\end{proof}
	
	\begin{cor}\label{GBoundarytoNBoundaryBigO}
		
		If $S_i$ is a sequence either of unions of special boxes, or of boxes \newline $[a_{0,i},a_{1,i})\times[b_{0,i},b_{1,i})\times[-\log(b_{1,i}-b_{0,i})+c_{1,i}, \log(a_{1,i}-a_{0,i})-c_{2,i})$ with $c_{j,i}$ bounded, then \newline $\mu(\partial^G_r(S))=O(\bigl|\partial^N_{r_1}(S\cap N)\bigr|)$, with $r_1$ as before.	
		
	\end{cor}

	\subsection{Scaling in the Net} \label{SolScalingIsAffine} 
	
	Since quasi-isometries on $N$ give rise to quasi-isometries on $G$ and vice versa, we will classify scaling quasi-isometries on $N$ by way of a structure theorem for $QI(G)$. This is \cite{EFW1} Theorem 2.1, together with Definition 2.2 and the remarks following it. The proof is completed in \cite{EFW2}.
	
	\begin{thm}[\cite{EFW1,EFW2}] \label{EFWMainThm}
		A quasi-isometry $\overline{q}:G\to G$ has a companion quasi-isometry $\overline{q}_1$ so that $\overline{q}\approx \overline{q}_1$ with $\overline{q}_1=L_g\circ (f_1,f_2,id)\circ \sigma$ for $L_g$ a left translation by a group element, $f_1, f_2$ Bi-Lipschitz maps of $\R$, and $\sigma\in S_2$ acting by possibly swapping the coordinate axes and negating the $t$-coordinate.
	\end{thm} 
	
	Notice that this finite-distance statement is unchanged if we round $\overline{q}_1$ to a map from $N$ to $N$. The purpose of this subsection is to use this theorem result to prove the following proposition.
	
	\begin{prop}\label{SolScalingIsAffineProp}
		Let $q:N\to N$ be a scaling quasi-isometry, inducing a map $\overline{q}:G\to G$ with companion map $\overline{q_1}$ in the sense of the previous theorem. Then the coordinate functions $f_1$ and $f_2$ of $\overline{q_1}$ are affine.
	\end{prop}
	
	The first step of this proof is to show that it suffices to consider maps of a specific type.
	
	\begin{lemma} \label{ReducingToProduct}
		
		With notation as above, there is map $q':N\to N$ with $q'=\rho\circ \overline{q'}$, and $\overline{q}':G\to G$ given by $(f_1', f_2', id)$ so that $q$ differs (up to $\approx$) from $q'$ by a composition of coarsely one-to-one maps and $f_i'$ is affine if an only if $f_i$ is.	
		
	\end{lemma}
	
	\begin{proof}
		
		Since $\rho\circ\overline{q}_1|_N\approx q$, we may assume  $q=\rho(\overline{q}_1|_N)$.
		Let $\overline{q}_1=L_g\circ(f_1,f_2,id)\circ \sigma$. Write \newline $g=x^by^ct^a$ Denote $x^{b}y^{c}=g'$.
		
		Since $\rho$ moves points a finite distance, we find that 
		
		$$q\approx \rho\circ L_{g'}|_N \circ \rho \circ L_{t^a} \circ (f_1, f_2, id)|_N\circ \rho \circ \sigma|_N=q_2$$
		
		Notice that $\sigma(e^{n_1}n_2,e^{-n_1}n_3,n_1)=(e^{-n_1}n_3,e^{n_1}n_2,-n_1)=t^{-n_1}x^{n_3}y^{n_2}$. As a result $\sigma$ restricts to a  map $\sigma|_N:N\to N$ that is bijective. Also, $\rho\circ\sigma|_N=\sigma|_N$. Therefore
		
		$$q_2\circ\sigma|_N=\rho\circ L_{g'}|_N \circ \rho \circ L_{t^a} \circ (f_1, f_2, id)|_N=q_3$$
		
		and $q_3$ differs from $q$ by a composition of $\approx$ and a one-to-one map.
		
		Consider a point $(b_1, c_1, n)=t^nx^{e^{-n}b_1}y^{e^nc_1}$ where $n$ is an integer. Then 
		
		\begin{align*}
			t^a\cdot(b_1,c_1,n)&=t^{a+n}x^{e^{-n}b_1}y^{e^nc_1}\\
			&=(e^ab_1,e^{-a}c_1,n+a)\\
			&=R_{t_a}\circ (h_1,h_2, id) (b_1,c_1,n)
		\end{align*}
		
		Where $h_1(x)=e^a x$ and $h_2(y)=e^{-a}y$.

		Since this applies to any point at integral height, it applies in particular for $(f_1, f_2, id)\cdot N$. Therefore, writing $f_i'=h_i\circ f_i$, we conclude that on $N$,
		
		$$q_3= \rho\circ L_{g'}|_N\circ\rho\circ R_{t^{a}}\circ (f_1', f_2', id)|_N$$
		
		But right multiplication by a power of $t$ commutes with the map $(f_1',f_2', id)|_N$, and any right multiplication moves points a bounded distance (in this case, a distance of $a$ since the vertical direction is undistorted in Sol). As such,
		
		$$q_3\approx\rho\circ L_{g'}|_N\circ\rho\circ (f_1', f_2', id)|_N$$
		
		Next, let $t^{n_3}x^{n_1}y^{n_2}=(e^{n_3}n_1, e^{-n_3}n_2, n_3)$ be a net point. Applying $\rho\circ L_{g'}$ to $t^{n_3}x^{n_1}y^{n_2}$ yields
		
		\begin{align*}
			\rho\circ L_{g'} (t^{n_3}x^{n_1}y^{n_2}) &=\rho(x^by^ct^{n_3}x^{n_1}y^{n_2})\\
			&=\rho(x^by^ct^{n_3}x^{n_1}y^{n_2})\\
			&=\rho\bigl((b+e^{n_3}n_1), (c+e^{-n_3}n_2), n_3\bigr)
		\end{align*}
		
		That is, at each integer height $n_3$, $\rho\circ L_{g'}$ acts by a shift followed by rounding. Such a map is one-to-one on $N$.
		
		As a result, if we write $\overline{q}'=(f_1', f_2', id)$, and $q'=\rho\circ \overline{q}'|_N$ then
		
		$$q_3\approx \rho\circ L_{g'}|_N\circ q'$$
		
		so that $q_3$ is approximately a composition of $q'$ and one-to-one maps. Applying a coarse inverse to $\rho\circ L_{g'}|_N$ on the left shows that $q'$ is a composition of one-to-one maps with $q_3$, which is itself a composition of $q_2\approx q$ with a one-to-one map.
		
		Finally, note that $f_i'$ is a composition of a multiplication function with $f_i$, so that $f_i'$ is affine if and only if $f_i$ is.
		
	\end{proof}
	
	We will only prove Proposition \ref{SolScalingIsAffineProp} for $f_1$, since the since the proof for $f_2$ is equivalent. The following lemmas will therefore be asymmetric between $x$ and $y$, but the analogous statements with the roles of the variables reversed are true by the same arguments.
	
	We will henceforth assume $q=\rho\circ (f_1, f_2, id)$, and prove prove two preliminary lemmas about the images of certain F\o lner sequences under such quasi-isometries. Firstly, we show that the images are again F\o lner sequences
	
	\begin{lemma} \label{ImagesofFolnerSeqs}
		
		Denote $\lambda$ the Lebesgue measure on $\R$. Let $I\subset \R$ be an interval and $J_i\subset \R$ be a sequence of intervals whose lengths tend to $\infty$ so that $\lambda(I)\lambda(J_i)>1$. Then $S_i=I\times J_i\times [-\log(\lambda(J_i)),\log(\lambda(I)))$ is a F\o lner sequence. Let $\overline{q}$ be any $K$ quasi-isometry $\overline{q}:G\to G$ with $\overline{q}=(f_1, f_2, id)$, where $f_i$ are $K$ bi-Lipschitz. Then $\overline{q}(S_i)=S_i'$ is a F\o lner sequence.
		
	\end{lemma}
	
	\begin{proof}
		
		Let the sequence of sets $f_1(I)\times f_2(J_i)\times [-\log(\lambda (f_2(J_i))),\log(\lambda(f_1(I))))$ be denoted $T_i$. We will show $S_i'$ is a F\o lner sequence by comparing it to $T_i$.
		
		Since the $f_i'$ are $K$-bi-Lipschitz, they increase or decrease lengths by at most a factor of $K$. Hence, $-\log(\lambda(f_2(J_i)))\le -\log(\lambda(J_i))+\log(K)$, and $\log(\lambda(f_1(I)))\ge\log(\lambda(I))-\log(K)$. Also, it is immediate that $$|\log(\lambda(f_1(I))\lambda(f_2(J_i)))-\log(\lambda(I)\lambda(J_i))|\le 2\log(K)$$ 
		Therefore, as $i\to \infty$, $$\frac{\log(\lambda(f_1(I))\lambda(f_2(J_i)))}{\log(\lambda(I)\lambda(J_i))}\to 1$$ Therefore $\frac{\mu(S_i')}{\mu(T_i)}\to 1$ because $\mu\bigl[x_1, x_2)\times [y_1, y_2)\times [t_1, t_2)\bigr)=(x_2-x_1)(y_2-y_1)(t_2-t_1)$.
		
		On the other hand, $T_i$ is within the $\log(K)$-neighborhood of $S_i'$ and vice versa, as the vertical direction is undistorted in $G$. Then any point within $r$ of $T_i$ is within $r+\log(K)$ of $S_i'$ and vice versa. The same holds for the neighborhoods of the complements of both sets. It follows that 
		
		$$\partial^G_r(S_i')\subset \partial^G_{r+\log(K)}(T_i)$$  
		
		so that 
		
		$$\lim_{i\to\infty}\frac{\mu\left(\partial^G_r(S_i')\right)}{\mu(S_i')}\le \lim_{i\to\infty}\frac{\mu\left(\partial^G_{r+\log(K)}(T_i)\right)}{\mu(S_i')}=\lim_{i\to\infty}\frac{\mu\left(\partial^G_{r+\log(K)}(T_i)\right)}{\mu(T_i)}=0$$
		
	\end{proof}
	
	Next, we prove a criterion that lets us conclude that scaling values in $G$ and $N$ agree.
	
	\begin{lemma} \label{SolScalingisNetScaling}
		
		Let $\overline{q}=(f_1,f_2, id)$ be a $K$ quasi-isometry of $G$ and let $q=\rho\circ\overline{q}$ be a $K$ quasi-isometry of $N$. Let $S_i$ be a F\o lner sequence as in Lemma \ref{ImagesofFolnerSeqs}, with image $S_i'$, and $R_i'=S_i'\cap N$. Then if $\overline{q}$ is $\gamma$-scaling on $S_i'$, $q$ is $\gamma$-scaling on $R_i'$.
		
	\end{lemma}
	
	\begin{proof}
		
		Firstly, note that by Corollary \ref{NAmenabilityCor}, $R_i'$ is indeed a F\o lner sequence so that the statement makes sense. By removing some of the leading terms of the sequence $S_i$ if necessary, we may assume that \newline $[-\log(\lambda(I))+\log(K), \log(\lambda(J_i))-\log(K))$ contains an integer for all $i$. Since $f_1$ and $f_2$ are $K$ bi-Lipschitz this implies that there is an integral height in $R_i'$.
		
		By assumption,
		
		$$|\mu(S_i)-\gamma\mu(S_i')|\le C\mu(\partial^G_r(S_i'))$$
		
		Then we can use Lemma \ref{GBoundarytoNBoundaryBox} and Corollary \ref{GBoundarytoNBoundaryBigO}, so that 
		
		$$|\mu(S_i)-\gamma\mu(S_i')|=O(|\partial^N_{r}(R_i')|)$$
		
		We then wish to bound $\bigl|\mu(S_i)-|q^{-1}(R_i')|\bigr|$ and $\bigl|\mu(S_i')-|R_i'|\bigr|$ in terms of $|\partial^N_r(R_i)|$
		
		By Lemma \ref{VolumePtCountComparison}, $\bigl| |R_i|-\mu(S_i')\bigr|\le\mu(\partial^G_{D_B} (S_i'))$. By the argument above, we apply \ref{GBoundarytoNBoundaryBox} and Corollary \ref{GBoundarytoNBoundaryBigO} to deduce that
		
		$$\bigl||R_i'|-\mu(S_i')\bigr|=O\bigl(|\partial^N_{r_1}( R_i')|\bigr)$$
		
		Therefore, 
		
		$$\bigl|\mu(S_i)-\gamma |(R_i')|\bigr|=O\bigl(|\partial^N_{r}( R_i')|\bigr)$$
		
		It remains to count $|q^{-1}(R_i')|$. If a point $p\in N$ is mapped by $q$ into $R_i'=\overline{q}(S_{i})\cap N$, then since $q=\rho\circ \overline{q}$,	 $d_G(\overline{q}(p), \overline{q} (S_i))\le D_B$. Hence $p$ is in the $KD_B+K$ neighborhood of $S_i$ by the quasi-isometric inequality. Therefore,
		
		$$q^{-1}(R_i')\subset \left(S_i\cup\partial^G_{KD_B+K}(S_i)\right)\cap N$$
		
		Similarly, since $\overline{q}$ is a $K$ quasi-isometry, a point $p$ more than $KD_B+K$ from $G\setminus S_i$ is mapped under $\overline{q}$ more than $D_B$ from $\overline{q}(G\setminus S_i)$. Since $\overline{q}$ is additionally bijective, $\overline{q}(G\setminus S_i)=G\setminus\overline{q}(S_i)$. As a result $\rho\circ\overline{q}(p)\in \overline{q}(S_i)$. It follows that $$\left(S_i\setminus \partial^G_{KD_B+K}(S_i)\right)\cap N\subset q^{-1}(R_i)$$
		
		Applying the inequalities in Remark \ref{ExtAndInt} and Lemma \ref{BoundaryAdditivity}, we deduce the upper bound 
		
		\begin{align*}
			|q^{-1}(R_i')|&\le \bigl|\left(S_i\cup\partial^G_{KD_B+K}(S_i)\right)\cap N\bigr|\\
			&\le \mu\Big(\ext{\left(S_i\cup\partial^G_{KD_B+K}(S_i)\right)}\Big)\\
			&\le  \mu\left(S_i\cup\partial^G_{(K+1)D_B+K}(S_i)\right)\\
			&\le \mu(S_i)+\mu\left(\partial^G_{(K+1)D_B+K}(S_i)\right)
		\end{align*}
		
		And we deduce the lower bound
		
		\begin{align*}
			\mu(S_i)-\mu\bigl(\partial^G_{(K+1)D_B+K}(S_i)\bigr)&\le  \mu\bigl(S_i\setminus \partial^G_{(K+1)D_B+K}(S_i))\bigr)\\
			&\le \mu\bigl(\int{(S_i\setminus\partial^G_{KD_B+K}(S_i))}\bigr)\\
			&\le \bigl|S_i\setminus \partial_{KD_B+K}^G(S_i)\cap N\bigr|\\
			&\le |q^{-1}(R_i')|
		\end{align*}
		
		Therefore, $$\bigl|\mu(S_i)-|q^{-1}(R_i')|\bigr|=O(\mu(\partial^G_r(S_i)))$$ 
		As with the term $\mu(\partial^G_{D_B} (S_i'))$, we convert to an error term in  $O(|\partial^N_{r_1}(S_i\cap N)|)$ by applying Lemma \ref{GBoundarytoNBoundaryBox} and Corollary \ref{GBoundarytoNBoundaryBigO} with $c=0$. We finally need to show that  $|\partial^N_{r_1}(S_i\cap N)|=O(|\partial^N_{r_2}(R_i')|)$ for some $r_2$, in order to have all error terms of the desired form. 
		
		As in Lemmas \ref{NAmenabilityCor} and \ref{GBoundarytoNBoundaryBox}, any point $p$ in $\partial^N_{r_1}(S_i\cap N)$ is mapped by $\overline{q}$ to a point within $Kr_1+K$ of both $S_i$ and $G\setminus S_i'$. Thus it is mapped into $\partial^G_{Kr_1+K}(S_i')$. Rounding $\overline{q}(p)$ reaches a net point in $\partial^G_{Kr_1+K+D_B}(S_i')$. As $q=\rho\circ\overline{q}$  is a $K$ quasi-isometry, applying Remark \ref{QIBoundedtoOne} yields
		
		$$|\partial^N_{r_1}(S_i\cap N)|\le B_{K^2+K} |\bigl(\partial^G_{Kr_1+K+D_B}(S_i')\bigr)\cap N|$$

		We then use Lemma \ref{VolumePtCountComparison} and Lemma \ref{BoundaryAdditivity} to obtain the upper bound
		
		$$|\partial^N_{r_1}(S_i\cap N)|\le B_{K^2+K}\bigl(\mu(\partial^G_{Kr_1+K+D_B}(S_i'))+\mu (\partial^G_{Kr_1+K+2D_B}(S_i'))\bigr)$$
		
		As before, we then apply Lemma \ref{GBoundarytoNBoundaryBox} and Remark \ref{GBoundarytoNBoundaryBigO} so the $G$-boundary of $S_{i,j}'$ to determine that
		
		$$|\partial^N_{r_1}(S_i\cap N)|=O\bigl(|\partial^N_{r_2} (R_i)|\bigr)$$
		
		Combining with the above yields
		
		$$\bigl||q^{-1}(R_i')|-\gamma |(R_i')|\bigr|=O\bigl(|\partial^N_{r}(R_i')|\bigr)$$
		
	\end{proof}
	
	Before proving Proposition \ref{SolScalingIsAffineProp}, we record the following proposition as an additional corollary of the above lemmas.
	
	\begin{prop} \label{CalculatingScalingValues}
		Suppose $f_1$ and $f_2$ are $K$ bi-Lipschitz maps of $\R$, and let $\overline{q}=(f_1,f_2,id)$ be a $K$ quasi-isometry of $G$. Let $q=\rho\circ \overline{q}$ be quasi-isometry of $N$. Denote $\lambda$ the Lebesgue measure of $\R$. Suppose $I$ is an interval and $J_i$ is a sequence of intervals of lengths tending to $\infty$. If $\lambda(f_1(I))=\gamma_1\lambda(I)$, and there is some $D>0$ so that $|\lambda(f_2(J_i))-\gamma_2\lambda(J_i)|\le D$, then $N$ contains a F\o lner sequence $R_i'$ so that $q$ is $\frac{1}{\gamma_1\gamma_2}$ scaling on $R_i'$.
	\end{prop} 
	
	The condition that $|\lambda(f_2(J_i))-\gamma_2\lambda(J_i)|\le D$ is weaker than needed in the sequel. Note that $\frac{\lambda(f_2(J_i))}{\lambda(J_i)}\to \gamma_2$ is not a sufficient assumption as described in Remark \ref{ScalingImpliesRatio}. 
	
	\begin{proof}
		
		Denote $S_i=I\times J_i\times [-\log(\lambda(J_i)),\log(\lambda(I)))$ as before, $S_i'=\overline{q}(S_i)$, and $R_i'=S_i'\cap N$. $R_i'$ is a F\o lner sequence by Lemma \ref{ImagesofFolnerSeqs}, and by Lemma \ref{SolScalingisNetScaling}, it suffices to compute a scaling value on $S_i'$. Since the Haar measure of a box in $G$ agrees with its Lebesgue measure, we need to show that
		
		$$|\gamma_1\gamma_2\lambda(I)\lambda(J_i)\log(\lambda(I)\lambda(J_i))-\gamma_1\lambda(I)\lambda(f_2(J_i))\log(\lambda(I)\lambda(J_i))|\le C\mu(\partial^G_r(S_i'))$$
		
		Since the metric in Sol is given by $e^{-2t}dx^2+e^{2t}dy^2+dt^2$, and the box $S_i'$ is below height $\log(\lambda(I)))$, one may check directly that there is some $r=r(\lambda(I))$ so that $\partial^G_r(S_i')$ contains a box centered on one of the $xt$ faces of $S_i'$ whose width in the $y$-direction is at least $D$. Then $\mu(\partial^G_r(S_i'))\ge D \lambda(I) \log(\lambda(J_i)\lambda(I))$. Then $\overline{q}$ is scaling on $S_i'$ since $|\lambda(f_2(J_i))-\gamma_2\lambda (J_i)|\le D$, with scaling value $\frac{1}{\gamma_1\gamma_2}$.
		
	\end{proof}
	
	Note that we make no assumption above that $\overline{q}$ or $q$ actually are scaling. In the presence of such an assumption, we would conclude that these makes are $\frac{1}{\gamma_1\gamma_2}$-scaling, because scaling maps scale by the same factor on every F\o lner sequence. When we wish to deduce that $\overline{q}$ or $q$ are scaling from the ground up, we will use the following lemma.
	
	\begin{lemma}\label{ScalingValuesofAffineQIs}
		
		Let $f_1$ and $f_2$ be $K$ bi-Lipschitz affine maps $f_1(x)=m_1x+c_1$, $f_2(y)=m_2y+c_2$, and let $q$ and $\overline{q}$ be the $K$ quasi-isometries in Proposition \ref{CalculatingScalingValues}. Then $\overline{q}$ is $\frac{1}{m_1m_2}$-scaling on Sol and $q$ is $\frac{1}{m_1m_2}$-scaling on $N$.
		
	\end{lemma}
	
	\begin{proof}
		
		Let $A\subset N$ be finite, and consider $\ext{A}\subset G$. Since $\overline{q}^{-1}$ induces a constant $\frac{1}{m_1m_2}$ on the Haar measure $\mu$, we compute
		
		\begin{align}
			\bigl|\mu(\overline{q}^{-1}(\ext{A}))-\frac{1}{m_1m_2}\mu(\ext{A})\bigr|&=0\\
			\bigl|\mu(\overline{q}^{-1}(\ext{A}))-\frac{1}{m_1m_2}|A|\bigr|&=0\\
			\bigl||(\overline{q}^{-1}(\ext{A}))\cap N|-\frac{1}{m_1m_2}|A|\bigr|&\le \mu\bigl(\partial^G_{D_B}(\overline{q}^{-1}(\ext{A}))\bigr)\\
			\bigl||q^{-1}(A)|-\frac{1}{m_1m_2}|A|\bigr|&\le \mu\bigl(\partial^G_{D_B}(\overline{q}^{-1}(\ext{A}))\bigr)\\
			\bigl||q^{-1}(A)|-\frac{1}{m_1m_2}|A|\bigr|&\le \mu\bigl(\overline{q}^{-1}(\partial^G_{KD_B+K}(\ext{A}))\bigr)\\
			\bigl||q^{-1}(A)|-\frac{1}{m_1m_2}|A|\bigr|&\le \frac{1}{m_1m_2}\mu\bigl(\partial^G_{KD_B+K}(\ext{A})\bigr)\\
			\bigl||q^{-1}(A)|-\frac{1}{m_1m_2}|A|\bigr|&\le \frac{1}{m_1m_2}\bigl|\partial^N_{r}(A)\bigr|
		\end{align}
		
		Where we used Lemma \ref{VolumePtCountComparison} in line 3, Lemma \ref{SurjectiveQIs} in line 5, and Lemma \ref{GBoundarytoNBoundarySpecial} in line 7.
		
	\end{proof}
	
	We next prove Proposition \ref{SolScalingIsAffineProp}, which is restated below.
	
	\noindent \textbf{Proposition \ref{SolScalingIsAffineProp}}	Let $q:N\to N$ be a scaling quasi-isometry, inducing a map $\overline{q}:G\to G$ with companion map $\overline{q_1}$ in the sense of Theorem \ref{EFWMainThm}. Then the coordinate functions of $\overline{q_1}$ are affine.
	
	\begin{proof}
		
		By Lemma \ref{ReducingToProduct}, we may assume that $q=\rho\circ (f_1, f_2, id)=\rho\circ \overline{q_1}$. We prove that $q$ being scaling implies $f_1$ is affine by proving the contrapositive. So let $f_1$ be a non-affine map, so that we have $I_1$, $I_2$ disjoint $x$-intervals of the same length $l_x$, but $f_1(I_1)$ and $f_1(I_2)$ are disjoint intervals of different lengths $l'_{x,1}$ and $l'_{x,2}$.
		
		Let $S_{i,j}=I_j\times [0, l_{y,i})\times[-\log(l_x),\log(l_{y,i}))$ where $j=1,2$ and $l_{y,i}\to \infty$, and with $l_{y,i}$ large enough that the interval $[-\log(l_x)+\log(K),\log(l_{y,i})-\log(K))$ contains an integer. By Lemma \ref{ImagesofFolnerSeqs}, $\overline{q}(S_{i,j})=S_{i,j}'$ are F\o lner sequences in $G$. Then by Corollary \ref{NAmenabilityCor}, the intersections $S_{i,j}'\cap N=R_{i,j}'$ are F\o lner sequences in $N$.
		
		We wish to show that for any given $k$, $q$ is not $k$-scaling on one of the sequences $R_{i,j}'$. By Lemma \ref{SolScalingisNetScaling}, we can consider the scaling value of $\overline{q}$ on the sequences $S_{i,j}'$. By Remark \ref{ScalingImpliesRatio}, then, it suffices to show that, if the limits exist, then $$\lim_{i\to\infty}\frac{\mu(S_{1,i})}{\mu(S_{1,i}')}\ne \lim_{i\to\infty} \frac{\mu(S_{2,i})}{\mu(S_{2,i}')}$$
		
		This follows from the direct computation
		
		\begin{align*}
			\lim_{i\to\infty}\frac{\mu(S_{i,j})}{\mu(S_{i,j})'}&=\lim_{i\to\infty}\frac{l_xl_{y,i}\log(l_{x,j}'l_{y,i}')}{l_{x,j}'l_{y,i}'\log(l_{x,j}'l_{y,i}')}\\
			&=\lim_{i\to\infty}\frac{l_xl_{y,i}}{l_{x,j}'l_{y,i}'}\\
			&=\frac{l_x}{l_{x,j}'}\lim_{i\to\infty}\frac{l_{y,i}}{l_{y,i}'}
		\end{align*}
		
		Since $l_{x,1}'\ne l_{x,2}'$, $\lim_{i\to\infty}\frac{\mu(S_{1,i}')}{\mu(S_{1,i})}= \lim_{i\to\infty} \frac{\mu(S_{2,i}')}{\mu(S_{2,i})}$ only if $\lim_{i\to\infty}\frac{l_{y,i}}{l_{y,i}'}=0$. But this is impossible since $f_2$ is $K$ bi-Lipschitz. It follows that if $f_1$ is not affine, $q'$ is not scaling. 
		
	\end{proof}
	
	Note that along the way we have proven that the same result is true for Sol.
	
	\begin{thm} \label{SolMeasureScalingIsAffineThm}
		
		Let $\overline{q}_1:G\to G$ be a quasi-isometry given by $L_g\circ (f_1, f_2, id)\circ \sigma$. If $\overline{q}_1$ is scaling (in the sense of measure) then the $f_i$ are affine.
		
	\end{thm}
	
	We conclude the subsection by showing that the of \ref{SolScalingIsAffineProp} is also true.
	
	\begin{prop}\label{SolAffineIsScalingProp}
		
		Let $q:N\to N$ be a $K$ quasi-isometry, with $q=\rho\circ L_{g}\circ \bigl((f_1, f_2, id)\circ\sigma\bigr)|_N=\rho\circ\overline{q}|_N$. If $f_1(x)=m_1x+b_1$ and $f_2(y)=m_2y+b_2$, then $q$ is $\frac{1}{m_1m_2}$-scaling.
		
	\end{prop}	
	
	\begin{proof}
		
		Let $R\subset N$ be a finite set. We need to show that there are positive constants $C$ and $r$ so that $\bigl||q^{-1}(R)|-\frac{|R|}{m_1m_2}\bigr|\le C|\partial^N_r(R)|$. 
		
		Since $\rho\circ\overline{q} =q$, $q^{-1}(R)= \overline{q}^{-1}(\ext{S})\cap N$. Since $q$ scales the Haar measure by a constant, 
		
		\begin{align*}
			\mu(\overline{q}^{-1}(\ext{S}))&=\frac{\mu(\ext{R})}{m_1m_2}\\
			&=\frac{|R|}{m_1m_2}
		\end{align*}
		
		Then $$\left| |q^{-1}(R)|-\frac{|R|}{m_1m_2}\right|=\left||\overline{q}^{-1}(\ext{R})\cap N|-\frac{\mu(\ext{R})}{m_1m_2}\right|$$
		
		By Lemma \ref{VolumePtCountComparison}, $$\bigl||\overline{q}^{-1}(\ext{R})\cap N|-\mu(\overline{q}^{-1}(\ext{R}))\bigr|\le \mu\bigl(\partial^G_{D_B} (\ext{R})\bigr)$$

		Combining the previous two lines with the fact that $\frac{\mu(\ext{R})}{m_1m_2}=\mu(\overline{q}^{-1}(\ext{R}))$, we observe
		
		$$\left| |q^{-1}(R)|-\frac{|R|}{m_1m_2}\right|\le\mu\bigl(\partial^G_{D_B}(\ext{R})\bigr)$$
		
		Applying Lemma \ref{GBoundarytoNBoundarySpecial}, we conclude 
		
		$$\left| |q^{-1}(R)|-\frac{|R|}{m_1m_2}\right|\le \bigl|\partial^N_{3D_B}(\ext{R}\cap N)\bigr|= \bigl|\partial^N_{3D_B}(R)\bigr|$$
		
		as required. 
	
	\end{proof}

	\subsection{The Rank-One Case of the Main Construction}\label{SolMainConstruction}
	
	Let $\gamma\in \R^+$. We will now start with the net $N$ and construct a space $N_\gamma$ with scaling group $\langle \gamma\rangle$ by attaching a collection of flats (coarsely) along the locus $xy^2=1$ in $N$. We then prove a technical result that any quasi-isometry must coarsely preserve both the $N$ subspace and this locus. The proof that $Sc(N_\gamma)=\langle \gamma \rangle$ will be delayed to the next subsection. Throughout this and the next subsection, we will take $\gamma>1$. Also, since $\gamma$ now refers to a specific number, we will henceforth refer to $k$-scaling when we wish to describe a map as scaling without specifying its value.
	
	\begin{Def}\label{DefinitionOfMainSpace}
		Let $N$ be as before, $\gamma\in \R^+$ be greater than $1$, and let $N_\gamma=\left(N\sqcup \bigsqcup_{i\in \Z} \Z^2\right)/\sim$, where $(0,0)_i\sim \rho(\gamma^{2i},\gamma^{-i},0)$ (denoting $(m,n)_i$ the point $(m,n)$ in the copy of $\Z^2$ indexed by $i$). We will define $a_i=\rho(\gamma^{2i},\gamma^{-i},0)$, and call $A=\{a_i\}$ as the set of \textit{attachment points}. Note that the $a_i$ need not all be distinct. See figure 3.
		
		\begin{figure}[H]
			\centering
			\includegraphics{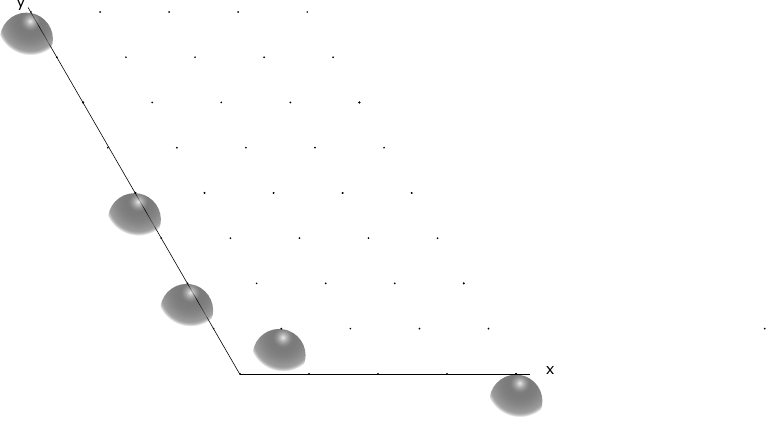}
			\caption{The height-$0$ slice of $N_2$, showing attachment points $-3$ through $1$. Attached planes are visualized as hemispheres.}
		\end{figure}
		
		We give $N_\gamma$ the metric $d_{N_\gamma}$ as follows. We endow $N\subset N_\gamma$ with the metric $d_N$ as before, and give each $\Z^2_i$ its taxicab metric $d_{\Z^2_i}$. Then $d_{N_\gamma}$ is the largest metric extending $d_{\Z^2_i}$ and $d_N$. That is, the distance between points $x\in \Z^2_i$ and $y\in N$ is $d_{N_\gamma}(x,y)=d_{\Z^2_i}(x,(0,0)_i)+d_N(a_i,y)$ and the distance between points $x\in \Z^2_i$ and $y\in\Z^2_j$ is $d_{N_\gamma}(x,y)=d_{\Z^2_i}(x,(0,0)_i)+d_N(a_i, a_j)+d_{\Z^2_j}((0,0)_j,y)$.
	\end{Def}
	
	We will denote $\partial_r$ to refer to boundaries as subsets of $N_\gamma$.
	
	We first make an elementary observation
	
	\begin{lemma} \label{FinitelyManyFlatsPerAttachment}
		Let $i$ be an integer. There are at most $\log_\gamma(2)+1$ integers $j$ so that $a_j=a_i$.
	\end{lemma}
	
	\begin{proof}
		If $\gamma\ge 2$, then $a_i$ has some coordinate that is at least $1$. Multiplying a positive integer by any (positive or negative) power of $\gamma$ and then rounding down will change its value, so that $a_i\ne a_j$ for any $i\ne j$.
		
		Then let $1<\gamma<2$, and suppose $i\le 0$. Then the $y$-coordinate of $a_i$ is a positive integer $n$. Choose the maximal $i$ among the collection of integers with attachment point $a_i$. Then if $a_j=a_i$, we must have $j\le i$, and $\lfloor \gamma^{-i}\rfloor=\lfloor \gamma^{-j}\rfloor$. It follows that $\gamma^{j-i}\le \frac{n+1}{n}$ so that $j-i< \log_\gamma\left(\frac{n+1}{n}\right)$. Clearly then there are at most $\log_\gamma(2)$ such $j$.
		
		The proof for $i>0$ is equivalent, but one uses the $x$-coordinate and takes $i$ minimal among the integers with attachment point $a_i$.
		
	\end{proof}
	
	Let us verify that $N_\gamma$ has a well-defined scaling group.
	
	\begin{lemma} \label{NgammaPropertiesLemma}
		$N_\gamma$ is UDBG and amenable.
	\end{lemma}
	
	\begin{proof}
		
		By the definition of $d_{N_\gamma}$, every distance between two points in $N_\gamma$ contains a summand that is either a $d_{\Z^2_i}$-distance or a $d_{G}|_N$-distance. Since both of these are uniformly discrete metrics, so is $d_{N_\gamma}$.

		For bounded geometry, first note that by Lemma \ref{FinitelyManyFlatsPerAttachment}, each point in $N$ meets at most $\log_\gamma(2)$ planes. 	Since both $N$ and $\Z^2$ have bounded geometry, denote the bound on the size of a $d_N$ r-ball to be $B_{N, r}$ and the bound on the size of a $d_{\Z^2}$ r-ball to be $B_{\Z^2, r}$. Now let $r>0$, and let $p\in N$. $B(p,r)$ contains an $N$-ball of size at most $B_{N, r}$. If all $B_{N, r}$ of those points are attachment points, the $B(p,r)$ meets the attachment points of at most $B_{N,r}\log_\gamma (2)$ planes, and contains at most an $r$-ball in each one of them.  Hence
		$$|B(p,r)|\le B_{N,r}+B_{N,r}\log_\gamma (2)B_{\Z^2, r}$$
		
		Now let $p$ be a point in $\Z^2_i\setminus \{(0,0)_i\}$, and $d_{\Z^2_i}(p, (0,0)_i)=d$. If $d>r$, then $B(p,r)$ is an r-ball in $\Z^2_i$ of size at most $B_{\Z^2, r}$. If $d<r$, then $B(x,r)$ is a union of an $r$-ball in $\Z^2_i$, and an $r-d$-ball about $a_i$. Then applying the previous paragraph to the $r-d$-ball abut $a_i$ gives
		
		\begin{align*}
			|B(p,r)|&\le B_{\Z^2, r}+|B(a_i, r-d)|\\
			&\le B_{\Z^2, r}+ B_{N,{r-d}}+B_{N,{r-d}}\log_\gamma (2)B_{\Z^2, {r-d}}\\
			&\le B_{\Z^2, r}+B_{N,r-1}+B_{N,r-1}\log_\gamma (2)B_{\Z^2, r-1}
		\end{align*}
		
		This proves bounded geometry.	
		
		To see amenability, take any F\o lner sequence $S_i$ of larger and larger squares in $\Z^2_i$ moving further and further from $(0,0)_i$. For any $r>0$, these squares are eventually more than $r$ away from $(0,0)_i$, so that their $r$-boundaries are exactly the $r$-boundaries of squares in $\Z^2$. Since $\frac{|\partial^{\Z^2}_r(S_i)|}{|S_i|}\to 0$, and $\partial_r(S_i)=\partial^{\Z^2}_r(S_i)$ for sufficiently large $i$, we conclude that $\frac{|\partial_r(S_i)|}{|S_i|}\to 0$
		
	\end{proof}
	
	For the remainder of this section, denote by $B_r$ the upper bound on the size of a $r$-ball in $N_\gamma$.
	
	We next state a preliminary result about quasi-isometries on $N_\gamma$.
	
	\begin{prop} \label{FlatPreservingProp}
		Let $q:N_\gamma\to N_\gamma$ be a $K$ quasi-isometry. Then $q(N)$ is at bounded distance from $N$, $q(\Z^2_i)$ is bounded distance from a single $q(\Z^2_j)$, and $q(A)$ is at bounded distance from $A$. All bounds depend only on $K$.
	\end{prop}
	
	The proof will rely on the following theorem, due independently to Kleiner-Leeb and to Eskin-Farb. The phrasing below is the one appearing in \cite{KleinerLeeb}, and is comparable to the statement of \cite{EskinFarb} Theorem 1.1.
	
	\begin{thm}\label{SolQuasiFlatsThm} [\cite{KleinerLeeb} 1.2.5]
		Let $\Phi:\R^2\to\h^2\times\h^2$ be a $K$ quasi-isometry. Then there is a number $D(K)$ so that $\Phi(\R^2)$ lies in the $D$-neighborhood of a union of at most $D$ maximal flats.	
	\end{thm}
	
	\begin{cor} \label{SolQuasiFlatsCor}
		There is no quasi-isometric image of $\R^2$ in Sol, and hence no quasi-isometric image of $\Z^2$ in $N$.
	\end{cor}
	
	\begin{figure}[h]\label{SolFlats}
		
		\centering
		\includegraphics[scale=0.9]{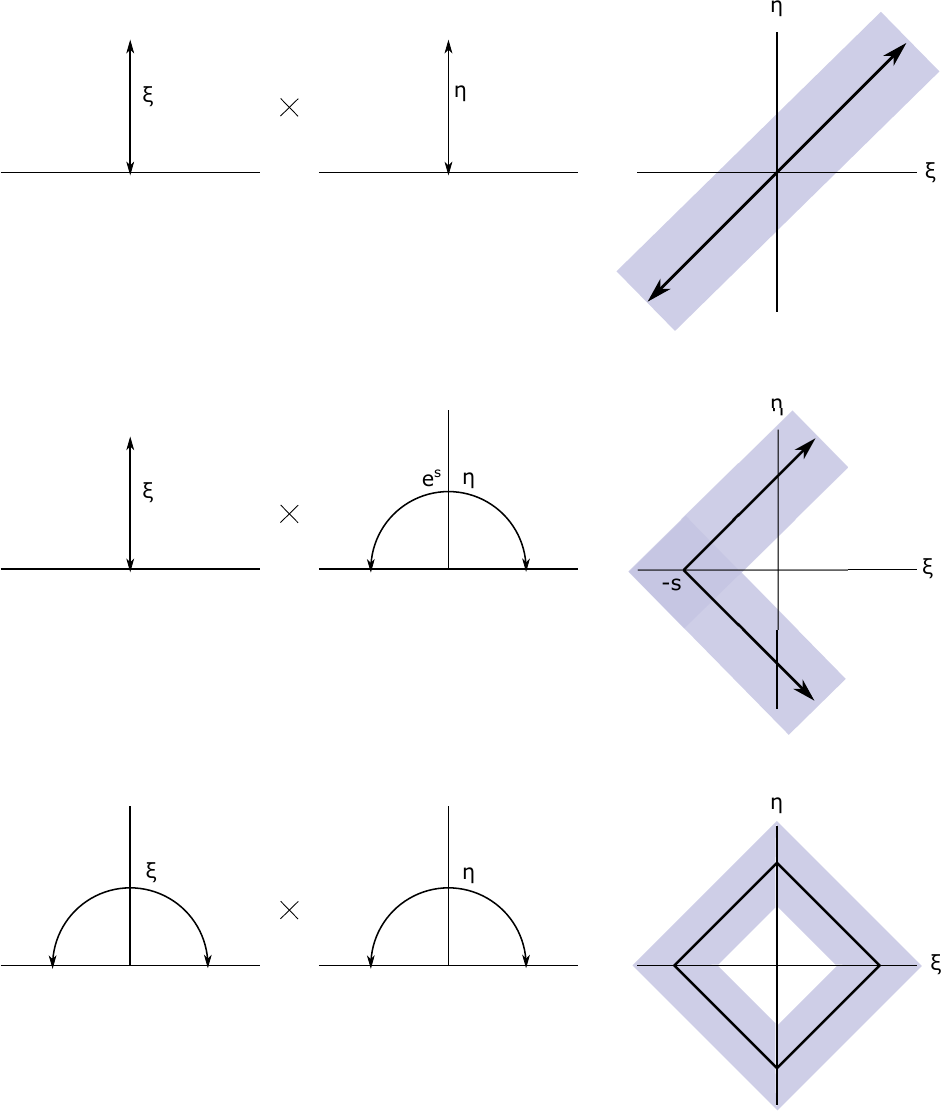}
		\caption{$\xi$ and $\eta$ are geodesics in the upper half plane model of $\h^2$. On the left are the cases where both are vertical (top), one is vertical (middle), or both are semicricles (bottom). Vertical geodesics are parameterized by height. On the right are the maximal flats $\xi\times\eta\subset\h^2\times\h^2$. The dark segments and rays are $(\xi\times\eta)\cap\Psi(G)$, and the shaded region is the intersection of $\mathscr{N}_D(\Psi(G))$ with this plane.}
		
	\end{figure}

	\begin{proof}
		Since $\R^2$ is quasi-isometric to $\Z^2$ and $N$ is quasi-isometric to Sol, we will instead prove that there is no quasi-isometric embedding of $\Z^2$ into Sol.
		
		By Remark \ref{FlatsInSolRk}, Sol quasi-isometrically embeds in $\h^2\times\h^2$. Let $\Psi$ be this quasi-isometric embedding, and recall that $\Psi(G)=\{((x,t),(y,s)):s+t=0\}$ where $t$ and $s$ are the logarithms of the heights in the upper half plane model. Composing with $\Psi$ pushes forward any quasi-flat in Sol to one in $\Psi(G)\subset \h^2\times\h^2$. Since $\Z^2$ and $\R^2$ are quasi-isometric, Theorem \ref{SolQuasiFlatsThm} applies to quasi-isometric images of $\Z^2$ in $\h^2\times\h^2$ as well.
		
		
		The maximal flats in $\h^2\times\h^2$ are products of geodesics in each coordinate. The intersections of these maximal flats with $\Psi(G)$ are shown in Figure 4, as well as the intersection $\mathscr{N}_D(\Psi(G))$ with a maximal flat. Moving a quasi-isometric image of $\Z^2$ in Sol by a distance at most $D$ would yield a copy of $\Z^2$ in a union of at most $D$ of these shapes. But the area growth of the shaded regions is at most linear (the flats $\xi\times\eta$ are isometrically embedded in $\h^2\times\h^2$), and therefore so is the area growth of a union of finitely many of them. Since $\Z^2$ grows quadratically, this is a contradiction. \end{proof}
	
	We are now prepared to prove Proposition \ref{FlatPreservingProp}.
	
	\begin{proof}
		
		$N$ has exponential growth due to the growth type of Sol, and $\Z^2$ has quadratic growth. Denote $E$ an exponential lower bound for the growth of $N$ and $Q$ a quadratic upper bound on the growth of $\Z^2$. $q$ is at most $B_{K^2+K}$-to-1 on any point. If a point $x\in N$ has $d(q(x),N)>D$, let $q(x)\in \Z^2_i$, and consider $B(q(x),D)\subset \Z^2_i$. By the quasi-isometric inequality, $q(B(x,\frac{D}{K}-K))\subset B(q(x),D)\subset\Z^2_i$ so that $E(\frac{D}{K}-K)\le B_{K^2+K}Q(D)$. This cannot be true for arbitrarily large $D$ because of the growth types on the left and right. Therefore, $q(N)$ is within bounded distance of $N$, with bound depending only on $K$.
		
		We next prove that $q(\Z^2_i)$ lies at bounded distance from exactly one $\Z^2_j$. By Corollary \ref{SolQuasiFlatsCor}, $q(\Z^2_i)$ cannot be within bounded distance of $N$. So there is a point of $\Z^2_i$ mapped at least $\frac{5}{2}K^3+K^2+K$ away from $N$, thus in some $\Z^2_j$. Then a ball of radius at least $\frac{5}{2}K^2+1$ is mapped into $\Z^2_j$ by the quasi-isometric inequality. Thus if $q^{-1}(\Z^2_j)\cap\Z^2_i$ is bounded, it contains a pair of points at distance at least $2\lfloor\frac{5}{2}K^2+1\rfloor> 5K^2$, that are adjacent to points mapped off of $\Z^2_j$. Call these points $(m_1,n_1)_i$ and $(m_2,n_2)_i$ and the adjacent points mapped off of $\Z^2_j$ $(m_1',n_1')_i$ and $(m_2',n_2')_i$.
		
		By the quasi-isometric inequality, $(m_1',n_1')_i$ and $(m_2',n_2')_i$ are mapped at most $2K$ away from the images of $(m_1,n_1)_i$ and $(m_2,n_2)_i$, which are themselves in $\Z^2_j$. Then the images of $(m_1,n_1)_i$ and $(m_2,n_2)_i$ lie within $B((0,0)_j,2K)$ since they lie within $2K$ of points in $N$. But $(a_1,b_1)_i$ and $(a_2,b_2)_i$ are at distance at least $5K^2$ apart, and thus must be mapped at least $4K$ apart, which is a contradiction. 
		
		It follows that an unbounded subset of $\Z^2_i$ is mapped into $\Z^2_j$. Let $q(m_1, n_1)_i\in\Z^2_j$, and connect $(m_1,n_1)_i$ to $(0,0)_i=a_i$ by a geodesic, say $\eta$. Then adjacent points of $\eta$ are distance $1$ apart, so that sequential points of $q(\eta)$ are distance at most $2K$ apart. Then consider the finite set $q^{-1}(B(a_j, K))\cap \Z^2_i$. Since the only way to get from $N\setminus A$ to $\Z^2_j$ is through $a_j$, $q(\eta)$ meets $B(a_j, K)$, so that $q^{-1}(B(a_j, K))\cap \Z^2_i$ is not empty.
		
		Because $\Z^2_i$ is 1-ended, the complement of any finite set consists of one unbounded component and finitely many bounded components. In the unbounded component, points can be connected by paths that miss $q^{-1}(B(a_j, K))\cap \Z^2_i$, and no such path crosses from $\Z^2_j$ to $N_\gamma\setminus\Z^2_j$ by the above argument. Since an unbounded subset of $\Z^2_i$ is mapped into $\Z^2_j$, we conclude the entire unbounded component of $\Z^2_i\setminus q^{-1}(B(a_j, K))$ is mapped into $\Z^2_j$. Therefore, $\Z^2_i$ is mapped within $D(i)$ to $\Z^2_j$, where $D(i)$ is the largest distance any of the points in the bounded components of $\Z^2_i\setminus q^{-1}(B(a_j, K))$ is mapped away from $\Z^2_j$.
		
		The second part of the proposition is proven if $D(i)$ is uniformly bounded. However, we have already shown that if any point in $\Z^2_i$ is sent more than $\frac{5}{2}K^3+K^2+K$ from $a_j$, then there would have to be an unbounded component of $\Z^2_i$ mapped into $N$. As before this would be a contradiction. So $D(i)\le \frac{5}{2}K^3+K^2+K$.
		
		For the third part of the proposition, note that attachment points $a_i$ are points that are on both $N$ and $\Z^2_i$. They are therefore mapped within bounded distance of both $\Z^2_j$ and of $N$. But the intersection of a neighborhood about $N$ with a neighborhood about $\Z^2_j$ is a ball about $a_j$, so that $a_i$ is mapped within bounded distance of $a_j$.
		
	\end{proof}
	
	The following corollary is immediate
	
	\begin{cor}\label{BijectivityOnFlats}
		Let $q:N_\gamma\to N_\gamma$ be a $K$ quasi-isometry. If $q(\Z^2_i)$ is within finite distance of $\Z^2_j$, denote $q_*(i)=j$. Then $q_*$ is a bijection.		
	\end{cor}
	
	\begin{proof}
		By one-endedness, $q_*(i)=j$ if and only if $q(\Z^2_i)\cap \Z^2_j$ is unbounded. If then $j_0$ is not in the image of $q_*$, $q(N_\gamma)\cap \Z^2_{j_0}$ is a union of uniformly bounded neighborhoods of $a_j$. Then points arbitrarily far from $a_j$ would be arbitrarily far from $q(N_\gamma)$ so that $q$ is not a quasi-isometry.
		
		This shows surjectivity. Injectivity follows from considering the coarse inverse of $q$.
	\end{proof}
	
	We conclude this section with an observation that will be necessary later.
	
	\begin{lemma} \label{BoundaryManyAttachmentPoints}
		Let $R\subset N$ be any finite set. Then $|R\cap A|\le |\partial^N_r(R)|$.
	\end{lemma}
	
	\begin{proof}
		
		Consider $\ext{R}=T$, and project onto the $xy$ plane to obtain a set $T_{xy}$ of area $\alpha$. $\alpha$ is a lower bound on the horizontal surface area of $T$, and an upper bound for the number of points in $R$ at height $0$.
		
		Then $\mu(\partial^G_1(T))\ge \alpha$ since the vertical direction is undistorted in $G$. By Remark \ref{ExtAndInt}, \newline $|\partial^G_{1+D_B}(R)\cap N|\ge \mu(\partial^G_1(T))$. But as $T=\ext{R}$, a point $p$ is in $T$ iff $\rho(p)\in T$. Rounding then shows that $\partial^G_{1+D_B}(T)\cap N\subset\partial^N_{1+2D_B}(R)$, providing the required upper bound on the number of points in $R$ at height $0$.
		
	\end{proof}

	\subsection{Proof of the Main Theorem in the Rank 1 Case}
	
	The purpose of this section is to prove the 1-generator case of Theorem \ref{MainThm}.
	
	\begin{thm} \label{SolMainThm}
		Let $N_\gamma$ be the space described in the previous subsection. Then the scaling group of $N_\gamma$ is $\langle \gamma\rangle$.
	\end{thm}
	
	Let $q:N_\gamma\to N_\gamma$ be scaling. We know already that $q$ restricts, up to bounded distance, to a map from $q|_N$ which induces a map $\overline{q|_N}:G\to G$. Denote the companion map $\overline{q_1|_N}=L_g\circ(f_1,f_2,id)\circ\sigma:G\to G$, and $q_1|_N=\rho\circ \overline{q_1}$, with $g=x^{s_1}y^{s_2}t^{s_3}$. Then we will consider the map $q_1:N_\gamma\to N_\gamma$ where \newline $q_1(p)=\begin{cases} q_1|_N(p) & p\in N \\ q(p) & p\not\in N\end{cases}$. Notice that $q_1\approx q$ so that it suffices to compute the scaling value of $q_1$. We will also write $\overline{q_1|_N}=L_{x^{s_1}y^{s_2}}\circ R_{t^{s^3}}\circ (f_1', f_2', id)\circ \sigma$ for new maps $f_i'$ that are affine if and only if the $f_i$ are, as in the previous section. We will assume $q$, and $q_1$ are $K$ quasi-isometric and $f_i, f_i'$ are $K$ bi-Lipschitz. This notation is fixed for the entire subsection. The proof will proceed as follows.
	
	First, we will show that $q_1|_N$ is scaling in order to invoke Proposition \ref{SolScalingIsAffineProp}. We will then observe that the $x$ and $y$ coordinates of the attachment points are unequally spaced, which will show that $\sigma$ is trivial. We will deduce that the linear terms of the maps $f_1'$ and $f_2'$ must be powers of $\gamma$. To conclude, we will use Proposition \ref{CalculatingScalingValues} to show that the scaling value of $q_1$, if it exists, must be a power of $\gamma$. This will show that $Sc(N_\gamma)\subset \langle \gamma\rangle$. We will then explicitly show that scaling value $\gamma$ is attained.
	
	\begin{prop} \label{ScalingRestrictsToNet}
		Let $q:N_\gamma\to N_\gamma$ and $q_1|_N:N\to N$ be as described above. Then if $q$ is $k$-scaling on $N_\gamma$, $q_1|_N$ is $k$-scaling on $N$.
	\end{prop}
	
	\begin{proof}
		Let $S\subset N$ be a finite subset. By assumption there are positive real numbers $r$ and $C$ so that
		
		$$\bigl||q^{-1}(S)|-k|S|\bigr|\le C |\partial_r(S)|$$
		
		We want the inequality
		
		$$\bigl||q_1|_N^{-1}(S)|-k|S|\bigr|\le C |\partial^N_r(S)|$$
		
		where $C$ and $r$ may change depending only on $K$ and $\gamma$. Since $q_1|_N\approx q|_N$, we will verify such an inequality instead for $q|_N$. In order for $q|_N$ to be $k$-scaling, we need the inequality
		
		$$\bigl||q^{-1}(S)\cap N|-k|S|\bigr|\le C |\partial^N_r(S)|$$
		
		Where again $C$ and $r$ may change depending on the previous choices of $C$, $r$, and on $K$ and $\gamma$. \newline $q^{-1}(S)\cap N=q|_N^{-1}(S)$ differs from $q^{-1}(S)$ because it is missing all the points on flats that are mapped to attachment points in $S$. By Lemmas \ref{FinitelyManyFlatsPerAttachment} and \ref{BoundaryManyAttachmentPoints}, there are at most $|\partial^N_r(S)|$ attachment points in $S$, each one attached to at most $ \log_\gamma(2)+1$ flats $\Z^2_i$. Since $q$ is a $K$ quasi-isometry, it is at most $Q(K^2+K)$-to-1 on the flats where $Q$ is the quadratic growth function of $\Z^2$. Since $q$ is injective on the collection of flats, it follows that $$\bigl| q^{-1}(S)-q|_{N}^{-1}(S)\bigr|\le C|\partial^N_r(S)|$$
		
		Similarly, $\partial_r(S)$ contains both $\partial^N_r(S)$ as well as balls about attachment points that are near $S$ but not near $N\setminus S$, which must therefore be in $S$. By Lemma \ref{BoundaryManyAttachmentPoints} we have at most $|\partial^N_r(S)|$ such attachment points, and $\partial_r(S)$ includes at most a ball of radius $r$ about each of them. By bounded geometry, we conclude that  
		
		$$\bigl||\partial^N_r(S)|-|\partial_r(S)|\bigr|\le C| \partial^N_r(S)|$$
		
		We therefore obtain the desired inequality
		
		$$\bigl||q|_N^{-1}(S)\cap N|-k|S|\bigr|\le C'|\partial^N_r(S)|$$
	\end{proof}
	
	As a result, Proposition \ref{SolScalingIsAffineProp} applies to $q_1$. We next show that $\sigma$ must be trivial.
	
	\begin{lemma} \label{SigmaIsTrivial}
		Let $\overline{q_1|_N}=L_g\circ (f_1,f_2,id)\circ\sigma$ as before, with the $f_i$ affine. Then $\sigma$ is trivial.
	\end{lemma}
	
	\begin{proof}
		
		We will prefer to consider  $\overline{q_1|_N}=L_{x^{s_1}y^{s_2}}\circ R_{t^{s^3}}\circ (f_1', f_2', id)\circ \sigma$, where $f_i'$ are again affine.
		
		By Proposition \ref{FlatPreservingProp}, no attachment point is sent more than some global bound $D_1=D_1(K)$ away from the attachment locus $A$. Since the vertical direction is undistorted, $D_1$ gives an upper bound on $|s_3|$.
		
		As described Proposition \ref{SolScalingIsAffineProp} $L_{x^{s_1}y^{s_2}}$ acts as a translation at any fixed height, in particular at height $s_3$. More precisely $x^{s_1}y^{s_2}(a, b, s_3)=(a+e^{s_3}s_1, b+e^{-s_3}s_2, s_3)$. 
		
		Denote by $x(p)$ the x-coordinate of $p$, and denote by $D_2$ the maximum number so that there are points $p_1\in G$ at height $s_3$ and $p_2$ at height $0$ with $d(p_1,p_2)<D_1$ $|x(p_1)-x(p_2)|<D_2$. This maximum exists since, as mentioned before, distances in Sol increase monotonically with respect to coordinate differences. By the above, for each integer $i$ there is an integer $j$ so that $|x(\overline{q_1|_N}(a_i))-x(a_j)|<D_2$, and therefore $|x\bigl( (f_1',f_2',id)\circ \sigma (a_i)\bigr)-x(a_j)|\le D_2+e^{s_3}s_1=D_3$.
		
		Since $f_1'$ is affine, there is a number $I$ is large enough that $\frac{f_1'(\lfloor\gamma^{2i+1}\rfloor)}{f_1'(\lfloor \gamma^{2i}\rfloor)}\in [.99\gamma, 1.01\gamma]$ for $i>I$. Also let $J$ be large enough that $|\gamma^{2j+2}-\gamma^{2j+1}|>|\gamma^{2j+1}-\gamma^{2j}|> .01\gamma^{2j+1}
		+ (1.01\gamma+2)(D_3)+1$ for $j>J$.
		
		Let $i>I$ be sufficiently large that $|f_1'(\lfloor \gamma^i\rfloor)-\lfloor\gamma^{2j}\rfloor|<D_3$ for some $j>J$. By assumption on $j$, the choice of such $j$ is unique. Now if $\sigma$ is nontrivial, then $(f_1', f_2', id)\circ \sigma(a_{-i})$ is within $D_3$ of $a_j$ and no other attachment point.
		
		Consider then the $x$-coordinate of $(f_1', f_2', id)\circ \sigma(a_{-i-1})$. By assumption, on $i$ and $j$
		
		\begin{align*}
			x((f_1', f_2', id)\circ \sigma(a_{-i-1})) &=f_1'(\lfloor \gamma^{i+1}\rfloor)\\
			&\in [.99\gamma f_1'(\lfloor \gamma^i\rfloor), 1.01\gamma f_1'(\lfloor \gamma^i\rfloor)]\\
			&\subset [.99\gamma \bigl(\lfloor\gamma^{2j}\rfloor-D_3\bigr), 1.01\gamma\bigl(\lfloor\gamma^{2j}\rfloor+D_3\bigr)]\\
			&\subset[.99\gamma^{2j+1}-.99\gamma(D_3+1), 1.01\gamma^{2j+1}+1.01\gamma(D_3)]
		\end{align*}
		
		By assumption, this interval is more than $D_3$ from the $x$-coordinate of any attachment point $\lfloor\gamma^{2j'}\rfloor$. This is a contradiction.	
			
		\end{proof}
		
		We are now prepared to prove that the linear terms of the $f_i'$ are powers of $\gamma$.
		
		\begin{prop} \label{PowersOfGamma}
			
			Let $\overline{q_1|_N}=L_{x^{s_1}y^{s_2}}\circ R_{t^{s^3}}\circ (f_1', f_2', id)$ as before, and $f_1'(x)=m_1x+c_1$ and $f_2'(y)=m_2 y + c_2$. Then $m_1=\gamma^{2n_1}$ and $m_2=\gamma^{n_2}$ for integers $n_i$.
			
		\end{prop}
		
		\begin{proof}
			
			Let $D_1$, $D_2$, and $D_3$ be as in Lemma \ref{SigmaIsTrivial}. Then we have for all $i$ there is a $j_i$ with $|x(\overline{q_1}(a_i))-x(a_{j_i})|<D_3$ by Lemma \ref{SigmaIsTrivial}. As $i\to\infty$, $\frac{f_1'(x(a_i))}{x(a_i)}=\frac{f_1'(x(a_i))}{\lfloor \gamma^{2i} \rfloor}\to m_1$. Since $f_1'(x(a_i))$ is in a bounded interval around $\gamma^{2{j_i}}$, as we take $i\to \infty$ this converges to a ratio of even powers of $\gamma$.
			
			The argument for $m_2$ and $y$ is the same, except that $\frac{f_2'(y(a_i))}{y(a_{j_i})}\to m_2$ as $i\to-\infty$, and the same calculation reveals it to be a ratio of powers of $\gamma$.	
			
		\end{proof}
		
		\begin{cor}\label{ScalingIsPowerOfGamma}
			
			Let $\overline{q_1|_N}=L_{x^{s_1}y^{s_2}}\circ R_{t^{s^3}}\circ (f_1', f_2', id)$, where $f_1'(x)=\gamma^{2n_1}x+c_1$ and $f_2'(y)=\gamma^{n_2}y+c_2$. Then there is a F\o lner sequence in $N_\gamma$ on which $q_1$ is $\gamma^{-2n_1-n_2}$-scaling.
			
		\end{cor}
		
		\begin{proof}
			
			By Proposition \ref{ScalingRestrictsToNet}, we need only compute a scaling value for $q_1|_N$, and by Lemma \ref{SolScalingisNetScaling}, this will equate with the scaling value of $\overline{q_1|_N}=L_{x^{s_1}y^{s_2}}\circ R_{t^{s^3}}\circ (f_1', f_2', id)\circ \sigma$. That this scaling value is $\gamma^{-2n_1-n_2}$ is a direct application of Proposition \ref{CalculatingScalingValues}.
			
		\end{proof}
		
		We now complete the proof of Theorem \ref{SolMainThm}.
		
		\begin{proof}
			
			By Corollary \ref{ScalingIsPowerOfGamma}, $Sc(N_\gamma)\subset \langle \gamma\rangle$. To show the reverse inclusion, it suffices to construct an explicit quasi-isometry for which the scaling value is $\gamma$. To this end, let 
			
			$$q(p)=\begin{cases} \rho(\frac{x}{\gamma^2},\gamma y, t) & p=(x,y,t)\in N\\
				(\lfloor\frac{a}{\gamma}\rfloor, b)_{i-1} & p=(a,b)_i, (a,b)\ne (0,0)
			\end{cases}$$
			
			We first verify that $q$ is indeed a quasi-isometry. Coarse density is immediate from the definition. If $p_1$ and $p_2$ are on $\Z^2_i\setminus N$, then since $\Z^2_i$ is given its taxicab metric, $\frac{1}{\gamma} d(p_1, p_2)-1\le d(q(p_1), q(p_2))\le d(p_1,p_2)$. Note that it is also the case that $\frac{1}{\gamma} d((0,0)_i, p_1)-1\le d((0,0)_{i-1}, q(p_1))\le d((0,0)_i,p_1)$. 	If $p_1$ and $p_2$ are instead in $N$, then $\frac{1}{K}d(p_1,p_2)-K\le d(q(p_1,p_2))\le Kd(p_1,p_2)+K$ for some $K$, since it is a rounding of a quasi-isometry of Sol.
			
			For the remainder of the quasi-isometric inequalities, we will need to bound how far $q(a_i)$ is from $a_{i-1}$. $a_i=(\lfloor\gamma^{2i}\rfloor,\lfloor\gamma^i\rfloor, 0)$, so that $q(a_i)=(\lfloor \frac{\lfloor \gamma^{2i}\rfloor}{\gamma^2}\rfloor, \lfloor \gamma \lfloor \gamma^{-i}\rfloor\rfloor, 0)$. One then checks immediately that \newline $|x(q(a_i))-x(a_{i-1})|\le \frac{1}{\gamma^2}+1$, while $|y(q(a_i))-y(a_{i-1})|\le \gamma+1$, so that $$d(q(a_i),a_{i-1})\le d(id, (\frac{1}{\gamma^2}+1, \gamma+1,0))=D$$
			
			The rest of the quasi-isometric inequality is as follows. If $p_1\in \Z^2_i\setminus N$ and $p_2\in \Z^2_j\setminus N$, then $$d(p_1,p_2)=d(p_1,(0,0)_i)+d(a_i, a_j)+d(p_2,(0,0)_j)$$ while
			$$d(q(p_1),q(p_2))=d(q(p_1),(0,0)_{i-1})+d(a_{i-1},a_{j-1})+d((0,0)_{j-1}, q(p_2))$$ By adding or subtracting at most $2D$, we see that
			
			$$|d(q(p_1),q(p_2))-d(q(p_1),(0,0)_{i-1})+d(q(a_i), q(a_j))+d((0,0)_{j-1}, q(p_2))|\le 2D$$
			
			We then apply the quasi-isometric inequalities we have on $N$ and $\Z^2_i$ to the above. If instead $p_1\in \Z^2_i\setminus N$ and $p_2\in N$, then the same argument works, replacing $d(a_i,a_j)+d(p_2, (0,0)_j)$ with $d(a_i,p_2)$, and $d(a_{i-1},a_{j-1})+d((0,0)_{j-1},q(p_2))$ with $d(a_{i-1},q(p_2))$.
			
			So $q$ is a  quasi-isometry, and we wish to prove its scaling value is $\gamma$. Let $R$ be a finite set, and consider $q^{-1}(R)$. By Lemma \ref{ScalingValuesofAffineQIs}, we know already that
			
			$$\bigl| |q^{-1}(R\cap N)\cap N| -\gamma |R\cap N|\bigr|\le C\big|\partial_r(R)\big|$$
			
			Since net points are only ever sent to net points by $q$, this implies that 
			
			$$\bigl| |q^{-1}(R)\cap N| -\gamma |R\cap N|\bigr|\le C\big|\partial_r(R)\big|$$
			
			We now wish to compute $|q^{-1}(R)\setminus N)|$. Call $R_i=R\cap \Z^2_i$. Decompose $R_i$ into strips $R_i=\bigsqcup_j [a_{1,j}, a_{2,j}]\times \{b_j\}\subset \Z^2_i$. Denote the $j^{th}$ strip $S_j$. Choose this collection maximally, i.e. so that  $[a_{1,j}-1,a_{2,j}]\times\{b_j\}\not\subset R_i$ and $[a_{1,j},a_{2,j}+1]\times \{b\}\not\subset R_i$. Then each $(a_{1,j}-1,b_j)_i$ is in $\partial_1(R_i)$.  If $(0,0)_i\not\in S_j$, then the pre-image is $q^{-1}(S_j)\setminus N=q^{-1}(S_j)=[\lceil\gamma a_{1,j}\rceil, \lceil \gamma a_{2,j} + \gamma\rceil -1]\times\{b\}\subset \Z^2_{i+1}$ by an elementary computation. If $(0,0)_i$ is in $S_j$, then its pre-image is almost the same, but the origin must be deleted because it is an element of $N$. There are $a_{2,j}-a_{1,j}+1$ elements in the $j^{th}$ strip, and $\lceil \gamma a_{2,j}+\gamma\rceil-\lceil \gamma a_{1,j}\rceil$ elements in its pre-image, or one feder if the origin is in $S_j$. As a result, if the origin is not in $S_j$, then $\bigl||q^{-1}(S_j)-\gamma |S_j|\bigr|\le 1$, or if $S_j$ does contain the origin, this difference is at most $2$. Since there are as many points $(a_{1,j}-1,b_j)_i$ as there are total strips, and at least one additional one $(a_{2,j}+1, b_j)_i$ for some $j$, we conclude that
			
			$$\bigl|q^{-1}(R_i)\setminus N|-\gamma |R_i|\bigr|\le\partial^{\Z^2_i}_1(R_i)$$
			
			Where as before $\partial_r^{\Z^2_i}(S)=\mathscr{N}_r(S\cap \Z^2_i)\cap\mathscr{N}_r(\Z^2_i\setminus S)\cap \Z^2_i$. We do this simultaneously on each $\Z^2_i$ meeting $R$ so that 
			
			$$\bigl||q^{-1}(R)\setminus N|-\gamma |R\cap\bigcup_i\Z^2_i|\bigr|\le\sum_i\partial^{\Z^2_i}_1(R_i)\le\partial_1(R)$$
			
			We have now counted the entire pre-image of $R$, but we have overcounted the points of $R$ by counting both $a_i$ in the inequality
			
			$$\bigl| q^{-1}(R)\cap N -\gamma |R\cap N|\bigr|\le C\big|\partial_r(R)\big|$$
			
			and $(0,0)_i$ in the inequality
			
			$$\bigl||q^{-1}(R)\setminus N|-\gamma |R\cap\bigcup_i\Z^2_i|\bigr|\le\partial_1(R)$$
			
			Then we combine these inequalities to obtain
			
			$$ \bigl| |q^{-1}(R)|-\gamma |R|\bigr|\le C\partial_r(R)+|A\cap R|$$
			
			The result then holds by applying Lemma \ref{BoundaryManyAttachmentPoints}.
						
		\end{proof}
		
		We record an observation based on the previous argument that maps of the given form are, up to changing the action on the $\Z^2$ subspaces, the only scaling quasi-isometries. More precisely:
		
		\begin{cor}
			Suppose $q|N\approx \rho\circ L_{x^{s_1}y^{s_2}}\circ R_{t^{s_3}}\circ (f_1',f_2',id)$ as before, with  $f_1'(x)=\gamma^{2n_1}x+b_1$ and $f_2'(y)=\gamma^{-n_2}y+b_2$, we must have $n_1=n_2$.
		\end{cor}
		
		\begin{proof}
		
			As in Lemma \ref{SigmaIsTrivial}, let $D_3$ be the constant so that if $\Z^2_i$ is sent near $\Z^2_j$, then $|f_1'(x(a_i))-x(a_j)|<D_3$. Analogously, denote $D_4$ a constant so that $\Z^2_i$ being sent near $\Z^2_j$ implies $|f_1'(y(a_i))-y(a_j)|<D_4$. Note that $D_3$ and $D_4$ depend only on $q$. Since $|x(a_i)-x(a_{i\pm 1})|\to \infty$ as $i\to \infty$ and $|y(a_i)-y(a_{i\pm 1})|\to\infty$ as $i\to-\infty$, we choose integers $I_1$ and $I_2$ as follows. Make $I_1$ sufficiently large so that $i>I_1$ implies that $a_{i-n_1}$ is the only attachment point with $|x(q(a_j))-x(a_i)|<D_3$ and make $I_2$ a sufficiently large negative number so that $i<I_2$ implies that $a_{i-n_2}$ is the only attachment point with $|y(q(a_j))-y(a_i)|<D_4$. 
			
			Therefore, denoting $q_*$ as in Corollary \ref {BijectivityOnFlats}, we have $q_*(i-n_1)=i$ for $i>I_1$ and $q_*(i-n_2)=i$ for $i<I_2$. Since $q_*$ is additionally a bijection, we conclude that $q_*[I_2-n_2,I_1-n_1]=[I_2,I_1]$. Therefore, $n_1=n_2$.
			
		\end{proof}
		
		\section{More generators and higher rank}\label{HigherRankSection}
		
		\subsection{Nets in higher-rank solvable Lie groups}

		To construct a space with a multiple-generator scaling group, we will need to use a higher-rank analogue of Sol, namely a Lie Group of the form $G=\R^{2n}\rtimes\R^{2n-1}$. It will be convenient, however, to describe these groups in terms of their Lie algebras. We fix some terminology following \cite{Peng1}. We will refer to the factor of $\R^{2n}=\mathbf{H}$ and the factor of $\R^{2n-1}=\mathbf{A}$, where the Lie algebra $\mathfrak{g}$ of $G$ splits as a sum of two abelian Lie algebras $\mathfrak{g}=\mathfrak{a}\oplus\mathfrak{h}$. $\mathfrak{a}$ is the Cartan subalgebra, with basis denoted $\{a_i\}_{i=1}^{2n-1}$, and acts on $\mathfrak{h}$ by roots $\alpha_1=a_1$, $\alpha_{2n}=-a_{2n-1}$, and $\alpha_i=a_{i}-a_{i-1}$ for $1\le i \le 2n-1$. Each root has a one-dimensional root space whose generator is denoted $h_i$.
		
		The group $G$ is the exponential of this Lie algebra. We will denote $t_i=\exp(a_i)$ and $x_i=\exp(h_i)$. We denote the root spaces in the group $\langle x_i\rangle=V_i\subset \mathbf{H}$. We will abuse notation and apply the roots $\alpha_i$ to vectors $\mathbf{t}\subset \mathbf{A}$, when it is implicitly meant to apply them to the associated element of the Lie algebra. With this notation, if $v_i\in V_i$, then $\mathbf{t} v_i\mathbf{t}^{-1}=e^{\alpha_i(\mathbf{t})}v_i$. As with Sol, we assign the coordinates $(b_1,...b_{2n}, c_1, ...c_{2n-1})=x_1^{b_1}...x_{2n}^{b_{2n}} t_1^{c_1}... t_{2n-1}^{c_{2n-1}}$.
				
		Our main construction will be a higher-dimensional analogue of the rank-one case: adding flats at points spaced along loci $x_{2i-1}x_{2i}^2=1$. The proofs in this section will be more terse, as they are usually the direct analogues of those in the prior section.
			
		We first construct a desired set of F\o lner sequences. As before, \textit{box sets} will refer to products of parameter intervals in the coordinates $(b_1, ... b_{2n}, c_1, ... c_{2n-1})$. Since the sum of the roots is $0$, the group $G$ is unimodular and the Haar measure of this Lie group is $\prod_{i=1}^{2n} dx_i\prod_{j=1}^{2n-1}dt_j$.

		\begin{lemma} \label{HigherRankFolnerSeqs}
			
			For any $j$, there is a F\o lner sequence $S_i$, whose $\mathbf{H}$-coordinates are always boxes which grow only in the $b_j$ coordinate, and shrink in each coordinate $b_k$ for $k\ne j$.
			
		\end{lemma}
		
		Notice that these sets are not required to be boxes in the $\mathbf{A}$ coordinates. 
				
		\begin{proof}
			
			Notice that for any collection $I$ of $2n-2$ of the roots $\alpha_i$, there is full-dimensional cone $\mathbf{C}$ of $\mathbf{A}$ for which $\alpha_i(\mathbf{C})< 0$ for each of the $\alpha_i\in I$. So choose the $\mathbf{C}$ on which all roots other than $\alpha_j$ are negative. Let $\Omega$ be an open subset of $\mathbf{C}$, and let $s_i$ be a positive real sequence tending to $\infty$, and let $s_i\Omega$ denote the dilation of $\Omega$ by a factor of $r_i$. Take $B(s_i\Omega)=\bigl(\prod_{k=1}^{2n} [0, e^{\max(\alpha_k(s_i\Omega))}]\bigr)s_i\Omega=S_i$. $S_i$ is a F\o lner sequence by \cite{Peng1} Lemma 2.2.1. Since the $\alpha_k$ are linear, $\max(\alpha_k(s_i\Omega))=s_i\max(\alpha_(\Omega))$ which is a sequence of negative numbers of increasing magnitude for $k\ne j$.  
			
		\end{proof}
		
		\begin{rk}\label{HigherRankFolnerBoxes}
			Note that the only part of the intervals $[0,e^{\max(\alpha_k(s_i\Omega))}]$ that mattered was their length. Since left multiplication is an isometry, we could left multiply by a power (or sequence of powers) of $x_k$ to shift these F\o lner sequences to have any desired endpoints. Also, note that we could take $\Omega$ to be any closed box inside of $\mathbf{C}$, and in such a case $B(\Omega)$ is honestly a product of parameter intervals. We will work in this setting for convenience.
		\end{rk}
		
		As with the embedding $\text{Sol}\to \h^2\times \h^2$, there is also a quasi-isometric embedding of $G\to \prod_{i=1}^{2n} \h^2$ given as follows (see \cite{Peng1} Lemma 2.1.1 and Remark 2.1.2). We map $(b_1, ...b_{2n}, \mathbf{t})$ to $\prod_{i} (b_i, e^{\alpha_i(\mathbf{t})})$. That is, if we use the plane model, which is the upper half plane model with log-scaled $y$-axis, then $G$ is sent to the locus with total $y$-coordinate $0$ in $\prod_{i=1}^{2n} \h^{2n}$.
		
		If we start with a box $B=\prod [0,1)$, we can left translate $B$ to partition $G$ and call the orbit of the identity $N$. Remark \ref{ExtAndInt} holds essentially verbatim, and so do its consequences. As in Lemma \ref{NPropertiesLemma} and its preliminaries, $N$ is a UDBG net with exponential growth. As in Corollary \ref{NAmenabilityCor}  F\o lner sequences in $G$ restrict to F\o lner sequences in $N$. We again denote $\rho$ the rounding to $N$, so that we will be able to translate between quasi-isometries $q:N\to N$ and $\overline{q}:G\to G$ at cost moving points a bounded distance. 
		
		We also have an analogue of Lemmas  \ref{GBoundarytoNBoundarySpecial} and \ref{GBoundarytoNBoundaryBox} in the higher-rank case. One advantage to working with F\o lner sequences that are boxes (as described in Remark \ref{HigherRankFolnerBoxes}) is that the proof of the second of these is directly analogous to that of Lemma \ref{GBoundarytoNBoundaryBox}, without having to deal with general convex subsets of $\mathbf{A}$.
		
		\begin{lemma}
			
			For each $r>0$, there is some $r_1=r_1(r)$ so that if $S$ is any union of special boxes \newline $S=\bigsqcup_{i=1}^k B_{n_{1,i},n_{2,i},n_{3,i}}$, then $\mu(\partial^G_r(S))\le|\partial^N_{r_1}(S\cap N)|$. 
			
		\end{lemma} 
		
		\begin{lemma}
			
			Let $K>0$, and let $\Omega$ be a closed box in $\mathbf{A}$. Let $S_i=B(s_i\Omega)$ be as before, with each $s_i$ sufficiently large that $s_i\Omega$ is wider than $2\log(K)+1$ in each basis direction. Let $S_i'$ be a sequence of box sets whose $x_j$-coordinate intervals differ from those of $S_i$ by factors $c_{i,j}$ in $[\frac{1}{K}, K]$ (but with $s_i\Omega$ as their $\mathbf{A}$ coordinate). Then for any $r$ there is an $r_1(r,K)$ so that  $\mu(\partial^G_r (S_i'))\le|\partial^N_{r_1}(S_i'\cap N)|$.
			
		\end{lemma}
		
		\subsection{Structure of quasi-isometries}\label{HigherRankScalingIsAffine}
		
		Here we state the structure theorem for $QI(G)$ analogous to Theorem \ref{EFWMainThm}, and use it to prove that scaling maps on $N$ are close to those with affine coordinate maps.	
				
		\begin{thm}[\cite{Peng1,Peng2}] \label{PengMainThm}
			
			Let $q:G\to G$ be a quasi-isometry with $G$ non-degenerate, unimodular, and split abelian-by-abelian. Then $q$ is at finite distance from a \textit{companion quasi-isometry} of the form $q_1:G\to G$ where $q_1=L_g\
			\circ(f_1(x_{\sigma(1)}),...f_{2n}(x_{\sigma(2n)}), Id)$. Here $\sigma\in S_{2n}$ is a coordinate permutation of the basis of $\mathbf{H}$ together with an associated linear transformation of $\mathbf{A}$ (see \cite{Peng2} at the beginning of Section 2.1 for details). $L_g$ is a left-translation by an element of $G$, and the $f_j$ are Bilipschitz maps of $\R$.
			
		\end{thm}
		
		We remark that $G$ is both unimodular and non-degenerate by the discussion in the previous subsection. Unimodularity was discussed already, and non-degeneracy is clear since the roots $\alpha_i$ calculated before span $\mathfrak{a}^*$. As a consequence, we obtain a result analogous to Proposition \ref{SolScalingIsAffineProp}
		
		\begin{prop} \label{HigherRankScalingIsAffineProp}
			If $q:N\to N$ is a scaling quasi-isometry inducing $\overline{q}:G\to G$, and $\overline{q_1}:G\to G$ is its companion map given by $L_g\circ (f_1,...f_{2n},Id)\circ \sigma$. Then the $f_i$ are affine.		
		\end{prop}
		
		\begin{proof}
			
			We reduce to $q=\rho(f_1,...f_{2n}, Id)$ as in Lemma \ref{ReducingToProduct}. Any coordinate permutation is 1-to-1 on boxes. The left action of any product of integral powers of the $t_j$ is also 1-to-1 on boxes, and the left action of fractional powers of the $t_j$ transforms into a right action at the cost of modifying the $f_j$ in an affine way. Finally, the left action of any power of an $x_j$ is 1-to-1 by the same computation as in Lemma \ref{ReducingToProduct}. 
						
			One change needs to be made to the proof of Proposition \ref{SolScalingIsAffineProp}, because we do not control the widths of the sequences of boxes that the previous subsection provides us. Let $\lambda$ denote the Lebesgue measure on $\R$ as before. If $f_k$ is not affine, then we can find two intervals $I_1$ and $I_2$ in the $x_k$-coordinate have equal length $l$, but where $\lambda(f_k(I_1))>\lambda(f_k(I_2))$. Then it is straightforward to verify that for any length $l'<l$, there are subintervals $I_1'\subset I_1$ and $I_2'\subset I_2$, both of length $l'$, so that $\frac{\lambda(f_k(I_1'))}{\lambda(f_k(I_2'))}\ge \frac{\lambda(f_k(I_1))}{\lambda(f_k(I_2))}>1$. 
			
			Then following the construction of F\o lner sequences in Lemma \ref{HigherRankFolnerSeqs} and Remark \ref{HigherRankFolnerBoxes}, we first find subintervals $I_1'$ and $I_2'$ of $I_1$ and $I_2$ of length less than $1$ satisfying the conclusion of the previous paragraph. This is necessary because the F\o lner sequences in Lemma \ref{HigherRankFolnerSeqs} grow in any $x_j$ coordinate where their starting width is greater than $1$, and we need them to shrink or stay the same. We let $r_i$ be a sequence of real numbers going to $\infty$, and construct two sequences of boxes $S_{i,1}$ and $S_{i,2}$ as in Remark \ref{HigherRankFolnerBoxes}. We take these sequences to match in every coordinate except the $x_k$ coordinate, where at each stage we take the subintervals of $I_1$ and $I_2$ guaranteed by the previous paragraph. We can always arrange this by left-translating by an appropriate power of $x_k$.
			
			We push these forward to sequence $S_{i,1}'$ and $S_{i,2}'$ using an analogous argument to Lemma \ref{ImagesofFolnerSeqs}, and then compute the scaling value for each using the argument of Lemma \ref{SolScalingisNetScaling}. We observe that these scaling values, if they exist, are nonzero and differ by a factor that is not $1$ as in the proof of Proposition \ref{SolScalingIsAffineProp}.
			
		\end{proof}

		\subsection{Attaching Flats}\label{HigherRankMainConstruction}
		
		Let $\Gamma=\langle \gamma_1,...\gamma_n\rangle\subset \R^+$ be a group isomorphic to $\Z^n$. We construct a space $N_\Gamma$ with	 scaling group $\Gamma$ analogously to the main construction in section \ref{SolMainConstruction}.
		
		We will use one attachment locus per generator. Therefore, for $1\le i\le n$, let \newline $A_i=\{\rho(0,...0, \gamma_i^{2j}, \gamma_i^{-j}, 0, ...0)\}=\{a_{i,j}\}$, where the nonzero coordinates of $a_{i,j}$ appear in the coordinates $x_{2i-1}$ and $x_{2i}$. The collection of all $a_{i,j}$ are again the \textit{attachment points}, and $\bigcup_i A_i=A$. We make the following concession to simplify a later argument. We identify each $a_{i,j}$ with the origin in a copy of $\Z^{2n-1+i}$. This will let us use a growth argument to rule out coordinate permutation. This is in fact unnecessary, because the attachment loci have incompatible spacing as in Lemma \ref{SigmaIsTrivial}, but that argument is more complicated.
		
		We will give $N_\Gamma$ the analogous metric to $N_\gamma$, by declaring the distance between points in different flats to be the distance to attachment points plus the distance between attachment points.
		
		The results of this section have proofs analogous to their rank-1 versions.
		
		\begin{lemma}
			$N_\Gamma$ is amenable and UDBG.
		\end{lemma}
		
		\begin{thm}
			[\cite{KleinerLeeb} 1.2.5]
			Let $\Phi:\R^{2n}\to\bigl(\h^2\bigr)^{2n}$ be a $K$ quasi-isometry. Then there is a number $D(K)$ so that $\Phi(\R^{2n})$ lies in the $D$-neighborhood of a union of at most $D$ maximal flats.	
		\end{thm}
		
		\begin{cor} 
			There is no quasi-isometric image of $\R^{2n}$ in $G$ and therefore no quasi-isometric image of $\Z^{2n}$ in $N$.
		\end{cor}
		
		\begin{proof}
			We analyze cases of maximal flats in $\bigl(\h^2\bigr)^{2n}$ as in the case of $\bigl(\h^2\bigr)^2$. They are all products of one geodesic $\eta_i$ in each factor. 
			
			If at least one of the $\eta_i$ is vertical (as in the top two cases of figure \ref{SolFlats}, then the intersection $\Phi(G)\cap \prod_j\eta_j$ has dimension $2n-1$, since we can choose any coordinate in the $\eta_j$ for $j\ne i$, an then choose the (unique) coordinate on $\eta_i$ to meet $\Phi(G)$. A $D$-neighborhood of $\Phi(G)$ similarly meets $\prod_j\eta_j$ in a thickening of this codimension-1 locus. If instead none of the $\eta_i$ is vertical, then as in the bottom case of figure \ref{SolFlats}, $\Phi(G)$ meets $\prod_j\eta_j$ in a compact (or empty) subset, as does its $D$-neighborhood.
			
			As a result, a quasi-isometrically embedded copy of $\Z^{2n}$ in $G$ would push forward by $\Phi$ and then by moving points a distance of $D$ to a quasi-isometrically embedded copy of $\Z^{2n}$ in a union of at most $D$ regions, each with growth of order at most $2n-1$. This is impossible.
		\end{proof}
		
		We next prove the following proposition, analogously to Proposition \ref{FlatPreservingProp}.
		
		\begin{prop}\label{FlatPreservingPropHigherRank}
			If $q:N_\Gamma\to N_\Gamma$ is any quasi-isometry, then there is some $D$ depending only on the quasi-isometry constant and $n$ so that $q(N)$ is in the $D$-neighborhood of $N$ and $q(A)$ is in the $D$-neighborhood of $A$. 
		\end{prop}
		
		\begin{proof}
			
			Large portions of $N$ cannot be mapped into a flat, because $N$ grows exponentially and each flat grows polynomially. This bounds how far $N$ can be sent off of itself, and how far $A$ can be sent away from $N$.
			
			The embedding of $G$ into a product of hyperbolic spaces provides a maximal rank of quasi flats in $G$ using \cite{EskinFarb,KleinerLeeb}. Since the flats attached to $N_\gamma$ exceed this rank, they must be mapped by $q$ close to flats. If $\Z^{m_1}_{j_1}$ is sent into $\Z^{m_2}_{j_2}$, then note first that $m_1\le m_2$ by comparing growth rates. Equality follows from considering the coarse inverse.
			
			If the attachment point for $\Z^{m_2}_{j_2}$ is $a_{i,j_2}$, then we use the same argument as in Proposition \ref{FlatPreservingProp} to say that if a sufficiently large portion of $\Z^{m_1}_{j_1}$ is mapped to $N$, then it will have a boundary of large diameter that is mapped close to $a_{i,j_2}$, in violation of the quasi-isometric inequality.
			
		\end{proof}

		The following lemma arises by a calculation analogous to Lemmas \ref{FinitelyManyFlatsPerAttachment} and \ref{BoundaryManyAttachmentPoints}.
		
		\begin{lemma} \label{BoundaryManyAttachmentPointsHigherRank}
			If $S\subset N$ is a finite subset, then the number of attachment points in $S$ is at most $|\partial_r(S)\cap N|$.
		\end{lemma}

		\subsection{The Main Theorem in Full Generality}
		
		In this section, we assemble the main theorem. The only substantial difference between this and the rank-1 case will be in eliminating coordinate permutations.
		
		\begin{thm}
			$Sc(N_\Gamma)=\Gamma$
		\end{thm}
				
		\begin{proof}
			
			Let $q:N_\Gamma\to N_\Gamma$ be $k$-scaling, send $N$  to $N$ and flats to flats. Denote $\overline{q|_N}$ the map $q|_N$ induces on $G$, $\overline{q_1}$ the companion map to $\overline{q|_N}$, and $f_i$ the coordinate maps of $\overline{q_1}$. 
			
			As in Proposition \ref{ScalingRestrictsToNet}, we observe that $q|_N$ is scaling. To do this we take a finite subset $S\subset N$ and note that $q|_N^{-1}(S)$ differs from $q^{-1}(S)$ only by undercounting the pre-images of attachment points. There are at most $|\partial^N_r(S)|$ of these by Lemma \ref{BoundaryManyAttachmentPointsHigherRank}, each with bounded pre-image. Similarly, $\partial_r(S)$ differs from $\partial^N_r(S)$ only by removing some balls about attachment points near $S$. The number of such attachment points is at most $|\partial^N_r(S)|$, and the sizes of balls in flats and number of flats identified to a single attachment point are both bounded. Hence we have $|\partial_r(S)|\le C_1|\partial^N_{r'}(S)|$, so that the statement that 
			
			$$||q^{-1}(S)|-k|S||\le C|\partial_r(S)|$$ 
			
			implies that $$||q|_N^{-1}(S)|-k|S||\le C|\partial_{r_1}(S)|\le C'|\partial^N_{r'}(S)\cap N|$$ for some altered constant $C'$ depending only on $C$ and $r'$ depending only on $r$. Thus $q|_N$ is scaling, and the $f_i$ are affine by Proposition \ref{HigherRankScalingIsAffineProp}.
			
			$\overline{q_1}$ cannot swap any pair of coordinates $(x_{2i-1},x_{2i})$ for the same reason as in Lemma \ref{SigmaIsTrivial}. That is, if it did, it would need to send $A_i$ to itself. However, $a_{i,j}$ for $j>>0$ are spaced by factors of $\gamma_i^2$ in the $x_{2i-1}$ coordinate, and $a_{i,j}$ for $j<<0$ are spaced by factors of $\gamma_i$ in the $x_{2i}$ coordinate. Swapping the two would map $a_{i,j}$ for $j<<0$ arbitrarily fair from $A_i$, in violation of Proposition \ref{FlatPreservingPropHigherRank}.
			
			If $\sigma$ makes some other permutation, some coordinate $x_{2i_1}$ or $x_{2i_1-1}$ must be sent to an $x_{2i_2}$ or $x_{2i_2-1}$ for $i_2<i_1$. But then attachment points at large $x_{2i_1}$ or $x_{2i_1-1}$ coordinates are sent near only attachment points at large $x_{2i_2}$ or $x_{2i_2-1}$ coordinates. Eventually, the image of such a point therefore must be within $D$ of only attachment points adjoined to flats of dimension $\Z^{2n-i+i_2}$, so that $\Z^{2n-1+i_1}$ must be quasi-isometrically mapped within finite distance of a copy of $\Z^{2n-1+i_2}$. But this is impossible since balls in $\Z^{2n-1+i_1}$ grow as a higher-degree polynomial than those in $\Z^{2n-1+i_2}$.
			
			We observe as in Proposition \ref{PowersOfGamma} that coarsely preserving the $A_i$ requires that the linear terms of $f_{2i}$ and $f_{2i-1}$ must be multiples of $\gamma_i$, by considering the limits $\lim_{j\to \infty} \frac{x_{2i-1}(q(a_{i,j}))}{x_{2i-1}(a_{i,j})}$ and $\lim_{j\to -\infty} \frac{x_{2i}(q(a_{i,j}))}{x_{2i}(a_{i,j})}$. The scaling value of a quasi-isometry $q$ is therefore a product of powers of the $\gamma_i$, by considering some F\o lner sequence far from the attachment locus and applying an analogue of Proposition \ref{CalculatingScalingValues}.
			
			This shows that $Sc(N_\Gamma)\subset \Gamma$. To show equality, we construct quasi-isometries $q_i$ with scaling value $\gamma_i$ for each $i$. We do this as in the proof of \ref{SolMainThm}. That is $\gamma_i|_N$ is a rounding of a map that divides by  $\gamma_i^2$ on $x_{2i-1}$, multiplies by $\gamma_i$ on $x_{2i}$, sends $\Z^{2n-1+i}_{i,j}$ to $\Z^{2n-1+1}_{i,j+1}$, and divides by $\gamma$ on some coordinate of $\Z^{2n-1+i}$ (followed by rounding). The verification that $q_i$ is a quasi-isometry with scaling value $\gamma_i$ follows along the same lines as in the proof of Theorem \ref{SolMainThm}.
			
		\end{proof}
		
		As before, it is a consequence of this argument that scaling maps act as translations along the tails of the sets $A_i$, and therefore must translate the same distance to maintain bijectivity.
		
		\begin{cor}
			With notation as before, if $f_{2i-1}(x_{2i-1})=\gamma_i^{2n_1}x_{2i-1}+b_1$ and $f_{2i}(x_{2i})=\gamma_i^{n_2}x_{2i}+b_2$, then $n_1=-n_2$,
		\end{cor}	
		
		\section{Improvements to the construction}\label{misc}
		
		We conclude with a few minor improvements to the main construction. First of all, we describe how to do this entire construction in the setting of metric measure spaces.
		
		\begin{rk}
			
			Instead of using the net $N$ and flats $\Z^n$, we could have attached copies of $\R^n$ to the groups $G$ in the same way as in the main construction. If we call the resulting space $G_\Gamma$, and endow it piecewise with the Haar measure on $G$ and the Lebesgue measure on the flats, then $G_\Gamma$ becomes a metric measure space. These spaces have $\Gamma$ as their scaling groups by the same argument as in the case of $N_\Gamma$: scaling quasi-isometries must have affine coordinate maps, and all quasi-isometries must (coarsely) preserve the attachment locus. In fact the proofs are easier, since we do not need to estimate the errors that arise from rounding and discretizing. For instance, the continuous analogue of the map $q$ constructed in the proof of Theorem \ref{SolMainThm} is obviously a quasi-isometry that scales the measure by $\gamma$.
			
		\end{rk}
		
		Next we show that this construction could instead have been carried out with graphs.
		
		\begin{prop}
			
			The spaces $N_\Gamma$ constructed in the previous sections are bi-Lipschitz equivalent to graphs with combinatorial metric.
			
		\end{prop}
		
		\begin{proof}
			
			We will construct graphs whose vertex sets are $N$ and $N_\Gamma$, and show that the combinatorial distance on these graphs is quasi-isometric to the metrics given in the construction of $N$ and $N_\Gamma$. This will then imply the result, since a bijective quasi-isometry of UDBG spaces is a bi-Lipschitz equivalence.
			
			
			We start with the case of $N$, and recall that $D_B$ denotes the diameter of the boxes used to generate $N$. We will abuse notation slightly in suppressing dependence on rank. Then we form the $(1+2D_B)$-Rips graph on the set $N$, with respect to the induced metric from the ambient Lie group $G$. We denote this graph $X$, and its combinatorial metric $d_X$. 
			
			One direction of the quasi-isometry is straightforward. Given a geodesic edge path in $X$ of length $n$, denote the vertices it traverses $(p_i)_{i=0}^n$. Then by the definition of $X$, $d_G(p_i, p_{i+1})\le 1+2D_B$. Hence the concatenation of the $G$-geodesics from $p_i$ to $p_{i+1}$ gives a path in $G$ of length at most $n(1+2D_B)$ in $G$. Hence the identity map $id:(N, d_X)\to (N, d_G)$ is $1+2D_B$ Lipschitz.

			For the other direction, $p$ and $p'$ be points of $N$, and denote $n=\lfloor d_G(p,p')\rfloor$. Let \newline $\gamma:[0, d_G(p,p')]\to G$ be a length-minimizing geodesic between $p$ and $p'$ traversed at unit speed. Consider the sequence of net points $p=p_0, p_1, ...p_n, p_{n+1}=p'$ where $p_i=\rho(\gamma(i))$ for $0<i<n+1$.
			
			Since $\rho$ moves points by a distance at most $D_B$, and $d_G(\gamma(i), \gamma(i+1))=1$ for $0\le i <n$ by assumption, $d_G(p_i, p_{i+1})\le 1+2D_B$ so that there is an edge $[p_i, p_{i+1}]$ in the graph $X$. Since $d_G(p, p')<n+1$, $d_G(p_n, p_{n+1})<1$ and therefore the same holds for $i=n$. Hence there is an edge path in $X$ of length $\lceil d_G(p, p')\rceil$ between $p$ and $p'$. The identity map $id:(N, d_G)\to (N, d_X)$ is therefore a $(1,1)$ quasi-isometry as required.
			
			Now consider the space $N_\Gamma$. Form the graph $X$ on the subspace $N$ of $N_\Gamma$. For each flat, add the edges to make the flat into the Cayley graph of $\Z^n$ with respect to the standard generating set. Call the resulting graph $X_\Gamma$.
			
			So let $p$ and $p'$ be points of $N_\Gamma$, and consider their distances $d_{X_\Gamma}(p,p')$ and $d_{X_\Gamma}(p,p')$. By Definition \ref{DefinitionOfMainSpace}, we have that $d_{N_\Gamma}(p,p')$ splits into a sum of at most 3 terms: A $d_G$-distance in $N$ and two taxicab distances in flats. An edge path in $X_\Gamma$ between two flats must pass through the vertices of the attachment locus. Therefore any edge path in $X_\Gamma$ can be broken into at most 3 sub-paths: one in the $X$ subgraph and two along Cayley Graphs. Since the taxicab metrics on $\Z^n$ is the same as the Cayley Graph metric for the standard generating set, we determine that $N_\Gamma$ and $X_\Gamma$ ar Bi-Lipschitz with the same constants as $N$ and $X$ were. 
		\end{proof}
		
		Finally, we describe an alternative construction of the spaces $N_\Gamma$ for which the bound on the geometry (or equivalently the maximum degree of the bi-Lipschitz equivalent graph) depends only on the number of generators $\gamma_i$ of $\Gamma$, rather than also depending on the magnitudes of $|\log(\gamma_i)|$ as in Lemma \ref{FinitelyManyFlatsPerAttachment}.
		
		\begin{rk}
			
			Suppose $\Gamma=\langle \gamma_i\rangle$ for a minimal collection of $\gamma_i>1$. For each $\gamma_i$, take $m_i$ so that $\gamma_i^{m_i}>2$. Then instead of attaching flats along the locus $x_{2i-1}x_{2i}^2=1$, we could have attached along the locus $x_{2i-1}^{m_i}x_{2i}^{{m_i}+1}=1$, with attachment point $a_{i,j}$ at coordinate $\rho(0,...0,\gamma_i^{(m_i+1)j},\gamma_i^{-m_ij},0,...0)$. These points are spaced by factors of at least $\gamma_i^{m_i}>2$ in each coordinate, so that no two attachment points can round to the same net point (as was the case in Lemma \ref{FinitelyManyFlatsPerAttachment} when $\gamma>2$). An explicit quasi-isometry scaling by $\gamma$ would be given by dividing $x_{2i-1}$ by $\gamma^{m_i+1}$ and multiplying $x_{2i-1}$ by $\gamma^{m_i}$.
			
		\end{rk}

		\bibliographystyle{plain}
		\bibliography{Scaling_Groups}

\begin{thebibliography}{10}

\bibitem{BlockWeinberger}
Jonathan Block and Shmuel Weinberger.
\newblock Aperiodic tilings, positive scalar curvature, and amenability of
  spaces.
\newblock {\em Journal of the American Mathematical Society}, 5(4):907 -- 918,
  1992.

\bibitem{Dymarz1}
Tullia Dymarz.
\newblock Bijective quasi-isometries of amenable groups.
\newblock In {\em Geometric methods in group theory}, volume 372 of {\em
  Contemp. Math.}, pages 181--188. Amer. Math. Soc., Providence, RI, 2005.

\bibitem{Dymarz2}
Tullia Dymarz.
\newblock Bilipschitz equivalence is not equivalent to quasi-isometric
  equivalence for finitely generated groups.
\newblock {\em Duke Mathematical Journal}, 154(3), 2010.

\bibitem{DymarzNavas}
Tullia Dymarz and Andres Navas.
\newblock Non-rectifiable delone sets in sol and other solvable groups.
\newblock {\em Indiana Univ. Math. J.}, 67:89--118, 2018.

\bibitem{EskinFarb}
Alex Eskin and Benson Farb.
\newblock Quasi-flats and rigidity in higher rank symmetric spaces.
\newblock {\em J. Amer. Math. Soc.}, 10(3):653--692, 1997.

\bibitem{EFW1}
Alex Eskin, David Fisher, and Kevin Whyte.
\newblock Coarse differentiation and quasi-isometries i: spaces not
  quasi-isometric to cayley graphs.
\newblock {\em Annals of Mathematics}, 176:221--260, 2012.

\bibitem{EFW2}
Alex Eskin, David Fisher, and Kevin Whyte.
\newblock Coarse differentiation of quasi-isometries ii: Rigidity for sol and
  lamplighter groups.
\newblock {\em Annals of Mathematics}, 177:869--910, 2013.

\bibitem{GenevoisTessera1}
Anthony Genevois and Romain Tessera.
\newblock Asymptotic geometry of lamplighters over one-ended groups, 2021.

\bibitem{GenevoisTessera2}
Anthony Genevois and Romain Tessera.
\newblock Measure-scaling quasi-isometries, 2021.

\bibitem{Gromov}
Mikhael Gromov.
\newblock Asymptotic invariants of infinite groups.
\newblock {\em London Math. Soc. Lecture Note Ser.}, 182:1--295, 1993.

\bibitem{KleinerLeeb}
Bruce Kleiner and Bernhard Leeb.
\newblock Rigidity of quasi-isometries for symmetric spaces and {E}uclidean
  buildings.
\newblock {\em Inst. Hautes \'{E}tudes Sci. Publ. Math.}, (86):115--197 (1998),
  1997.

\bibitem{Nekrashevych}
Volodymyr Nekrashevych.
\newblock Quasi-isometric hyperbolic groups are bi-lipschitz equivalent.
\newblock {\em Dopov. Nats. Akad. Nauk Ukr. Mat. Prirodozn. Tekh. Nauki},
  (1):32--35, 1998.

\bibitem{NowakYu}
Piotr~W. Nowak and Guoliang Yu.
\newblock {\em Large scale geometry}.
\newblock EMS Textbooks in Mathematics. European Mathematical Society (EMS),
  Z\"{u}rich, 2012.

\bibitem{PapasogluWhyte}
Panos Papasoglu and Kevin Whyte.
\newblock Quasi-isometries between groups with infinitely many ends.
\newblock {\em Commentarii Mathematici Helvetici}, 77:133 -- 144, 2002.

\bibitem{Peng1}
Irine Peng.
\newblock {Coarse differentiation and quasi-isometries of a class of solvable
  Lie groups I}.
\newblock {\em Geometry and Topology}, 15(4):1883--1925, 2011.

\bibitem{Peng2}
Irine Peng.
\newblock Coarse differentiation and quasi-isometries of a class of solvable
  lie groups ii.
\newblock {\em Geometry and Topology}, 15:1927--1981, 2011.

\bibitem{Whyte}
Kevin Whyte.
\newblock Amenability, bilipschitz equivalence, and the von neumann conjecture.
\newblock {\em Duke Mathematical Journal}, 99(1):93 -- 112, 1999.

\end{thebibliography}
		
	\end{document}